\documentclass[a4paper]{article}
\usepackage{amsmath}
\usepackage{amsfonts}
\usepackage{amssymb}
\usepackage{graphicx}
\usepackage{verbatim}
\usepackage{amsmath, comment}
\usepackage{amssymb, mathrsfs}
\usepackage{amsfonts}
\usepackage{epic,gastex}
\usepackage{eucal}
\usepackage{array}
\usepackage{url}
\usepackage{color,colortbl}
\usepackage{wrapfig}
\usepackage{mathptmx}
\usepackage[all,cmtip]{xy}    
\newtheorem{thm}{Theorem}[section]
\newtheorem{lem}[thm]{Lemma}
\newtheorem{corollary}[thm]{Corollary}
\newtheorem{pro}[thm]{Proposition}
\newtheorem{example}[thm]{Example}
\newtheorem{definition}[thm]{Definition}
\newtheorem{remark}[thm]{Remark}
\newenvironment{proof}{\paragraph{Proof}}{\hfill$\square$}

\DeclareMathOperator{\p}{\mathrm{\textbf{p}}}
\DeclareMathOperator{\dau}{\mathrm{\textbf{d}}}
\DeclareMathOperator{\s}{\mathrm{\textbf{s}}}
\DeclareMathOperator{\rt}{\mathrm{\textbf{r}}}
\newcommand{\dom}{\mathop{\rm dom}\nolimits}
\newcommand{\im}{\mathop{\rm im}\nolimits}
\newcommand{\id}{\mathop{\rm id}\nolimits}
\newcommand{\lev}{\mathop{\rm lev}\nolimits}

\DeclareMathOperator{\K}{\mathscr{K}}
\DeclareMathOperator{\Lg}{\mathscr{L}}
\DeclareMathOperator{\R}{\mathscr{R}}
\DeclareMathOperator{\D}{\mathscr{D}}
\DeclareMathOperator{\On}{\mathcal{O}}

\DeclareMathOperator{\J}{\mathscr{J}}
\DeclareMathOperator{\Hg}{\mathscr{H}}
\DeclareMathOperator{\T}{\mathcal{T}}
\newcommand{\const}{\mathop{\rm c}\nolimits}
\def\institute#1{\gdef\@institute{#1}}

\immediate\write18{texcount -tex -sum  \jobname.tex > \jobname.wordcount.tex}

\providecommand{\keywords}[1]
{
  \small	
  \textbf{\textit{Keywords:}} #1
}

\title{\textbf{\textsc{$\R$-cross-sections of the monoid of  order-preserving transformations on a finite chain}} \thanks{Supported by the Ministry of Science and Higher Education of the Russian Federation (Ural Mathematical Center project No. 075-02-2022-877)}}
\author{Eugenija A. Bondar  \\
        \small Ural Federal University\\
        \small{e.a.bondar@urfu.ru}
}
\date{} 

\begin{document}
\maketitle

\begin{abstract}
We classify the $\R$-cross-sections of the monoid of order-preserving transformations on the $n$-element chain in terms of certain binary trees.
\end{abstract} \hspace{10pt}


\maketitle

\section{Introduction}

\label{intro}
Let $\rho$ be an equivalence relation on a semigroup $S$. A subsemigroup $S'$ of $S$ is called a $\rho$-\emph{cross-section} in $S$ if $S'$ contains exactly one representative of each $\rho$-class. In general, given $S$ and $\rho$, a $\rho$-cross-section of $S$ need not exist, and if it exists, it need not be unique. Therefore, admitting a $\rho$-cross-section with respect to an equivalence $\rho$, chosen to be tightly enough connected with the multiplication in $S$, imposes a non-trivial restriction on $S$. Studying semigroups under such restrictions may shed some light on the structure of $S$ as the $\rho$-cross-section $S'$ may be thought of as a skeleton of $S$ to which the flesh of $\rho$-classes is attached.

As Howie~\cite{How02} writes, ``on encountering a new semigroup, almost the first question one asks is `What are the Green relations like?'" Therefore, it is not a surprise that studying $\rho$-cross-sections has started from $\rho$ being one of these all-pervading  Green equivalencies $\Hg,\R,\Lg,\D,\J$. There are many publications devoted to $\K$-cross-sections, $\K\in\{\Hg,\R,\Lg,\D,\J\}$, in various transformation monoids and other semigroups with transparent Green structure; see Chapter 12 of \cite{Gan_Maz_book} and references therein.

The present paper deals with $\R$- and $\Lg$-cross-sections in the monoid $\On_n$ of all order-preserving transformations on the $n$-element chain. This is a very well-studied regular submonoid of the symmetric monoid $\T_n$; see \cite[Chapter 14]{Gan_Maz_book} and references therein for various structural and combinatorial aspects of $\On_n$. However, to the best of our knowledge, the $\R$- and $\Lg$-cross-sections of $\On_n$ have not yet been classified. The present paper closes this gap.

Since $\On_n$ is a regular submonoid of $\T_n$, its Green equivalencies $\R$ and $\Lg$ are nothing but the restrictions of the corresponding equivalencies on $\T_n$. One might therefore  expect that the $\R$- and $\Lg$-cross-sections in $\On_n$ are sort of restrictions of the $\R$- and $\Lg$-cross-sections of $\T_n$ whose descriptions were found by respectively Pekhterev \cite{Pekh2003} and the present author \cite{Bondar2014,Bondar2016}. We will see that this is true only for the $\Lg$-cross-sections. In contrast, for the $\R$-cross-sections, the situation turns out to be very different, and their classification requires certain novel ingredients. Surprisingly enough, a classical data structure from the theory of combinatorial algorithms, namely binary search trees, plays a crucial role here.

The paper is organized as follows. In Sect.~2 we first explain the situation with the $\Lg$-cross-sections in $\On_n$ and then provide a few ways to construct various $\R$-cross-sections in $\On_n$. In Sect. 3 we collect necessary definitions and notation from the theory of binary trees. In Sect. 4 we propose a construction in terms of certain binary trees that produces all $\R$-cross-sections of $\On_n$. In Sect. 5 we discuss  connections between $\Lg$- and $\R$-cross-sections of $\On_n$. Section 6 is devoted to the classification of $\R$-cross-sections of $\On_n$ up to isomorphism.

\section{$\Lg$- and $\R$-cross-sections of $\On_n$: first examples}
\label{sec:1}
\subsection{Preliminaries}
We assume the reader's acquaintance with a few basic concepts of semigroup theory, including the concept of Green's relations.

We denote by $\overline{n}$ the set $\{1,2,\dots,n\}$ and by $\T_n$ the monoid of all transformations on $\overline{n}$ acting on the right. Let $\overline{n}$ be ordered in a natural way. A transformation $\alpha\in\T_n$ is called \emph{order-preserving} if $x\leq y$ implies $x\alpha\leq y\alpha$ for all $x,y\in\overline{n}$. The set of all order-preserving transformations is a submonoid in $\T_n$ denoted by $\On_n$. A subsemigroup $R$ of $\On_n$ constitutes an $\R$-\emph{cross-section} of $\On_n$ if $R$ contains exactly one representative from each $\R$-class of $\On_n$. Dually, a subsemigroup $L$ of $\On_n$ is an $\Lg$-\emph{cross-section} of $\On_n$ if $L$ contains exactly one representative from each $\Lg$-class of $\On_n$. We write $\R(\On_n)$ and $\Lg(\On_n)$ for the set of all $\R$- and $\Lg$-cross-sections of $\On_n$, respectively. Analogously, we denote by $\R(\T_n)$ and $\Lg(\T_n)$ the sets of the $\R$- and respectively $\Lg$-cross-sections of  $\T_n$.

We write $\id_{\overline{n}}$ for the identity transformation on $\overline{n}$. For $x\in\overline{n}$, we write $\const_x$ for the constant transformation with the image $\{x\}$.
For a transformation $\alpha\in\mathcal{T}_n$, let $\im{(\alpha)}$ stand for its image and $\ker{(\alpha)}$ for its kernel. If $A_1,A_2,\dots,A_k$ are the $\ker{(\alpha)}$-classes of $\overline{n}$ and $A_1\alpha=a_1$, $A_2\alpha=a_2$, \dots, $A_k\alpha=a_k$, we represent $\alpha$ as follows:
   $$\alpha=\begin{pmatrix}
    A_1       & A_2      &  \dots & A_k \\
    a_1       &a_2       &  \dots & a_k \\
   \end{pmatrix}.$$

The Green equivalencies on $\On_n$ are known to be the restrictions of the corresponding equivalencies on $\T_n$; see \cite[Chapter 14]{Gan_Maz_book}, but be warned that transformations act on the left in \cite{Gan_Maz_book}. Thus, for all $\alpha,\beta\in\On_n$,
 \begin{itemize}
 \item[a)] $\alpha\R\beta$ if and only if $\ker{(\alpha)}=\ker{(\beta)}$;
 \item[b)] $\alpha\Lg\beta$ if and only if $\im{(\alpha)}=\im{(\beta)}$;
 \item[c)] $\alpha\Hg\beta$ if and only if $\ker{(\alpha)}=\ker{(\beta)}$ and $\im{(\alpha)}=\im{(\beta)}$;
 \item[d)] $\alpha\D\beta$ if and only if $\alpha\J\beta$ if and only if $|\im{(\alpha)}|=|\im{(\beta)}|$.
 \end{itemize}
\subsection{$\Lg$-cross-sections of $\On_n$}

For brevity, we say that two sets \emph{intersect} if their intersection is non-empty. We need the following observation.
\begin{lem}\label{lemma_cr_sec_of_reg_subsem}
Let $S$ be a semigroup, $T$ its regular subsemigroup, and $\K\in\{\Lg,\R\}$. A subset $C$ of $T$ is a $\K$-cross-section of $S$ if and only if $T$ intersects each $\K_S$-class and
$C$ is a $\K$-cross-section of $T$.
\end{lem}

\begin{proof}
$T$ being regular implies that $\K_T=\K_S\cap(T\times T)$; see, e.g., \cite[Lemma 1.2.13]{Higgins_book}.

For the `only if' part, if $C$ is a $\K$-cross-section of $S$, then by the definition of a $\K$-cross-section, $C$ intersects each $\K_S$-class and so does $T\supseteq C$. Further,  $C$ is a subsemigroup of $S$, and hence, of $T$ with exactly one element in each $\K_T$-class, so $C$ is a $\K$-cross-section of $T$.

For the `if' part, since the intersection of $T$ with any $\K_S$-class is non-empty, the intersection is a $\K_T$-class. Since $C$ is a $\K$-cross-section of $T$, it has exactly one element in this $\K_T$-class. Hence, $C$ exactly one element in each $\K_S$-class and is a $\K$-cross-section of $S$.
\end{proof}

Fig.~\ref{fig:lemma1} illustrates the lemma. The outer and the inner rectangles represent the semigroup $S$ and its subsemigroup $T$ respectively, the horizontal lines show the partition into $\K$-classes, and the dots symbolize the representatives that form the subsemigroup $C$ serving as a $\K$-cross-section in both $S$ and $T$.
 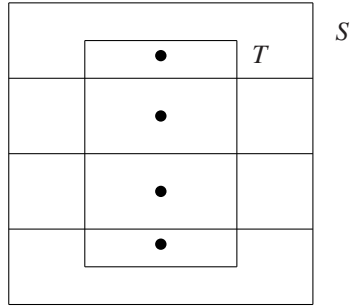
\begin{figure}[h]
\begin{center}
 \begin{picture}(40,40)
\gasset{AHnb=0,Nw=1.5,Nh=1.5,Nframe=n,Nfill=y}
\gasset{ExtNL=y,NLdist=1.5,NLangle= 60}
  \node(1)(20,8){}
  \node(2)(20,15){}
  \node(3)(20,25){}
  \node(4)(20,33){}
  \put(43,35){\mbox{$S$}}
  \drawline(0,0)(40,0)
  \drawline(0,10)(40,10)
  \drawline(0,20)(40,20)
  \drawline(0,30)(40,30)
  \drawline(0,40)(40,40)
  \drawline(0,0)(0,40) \drawline(40,0)(40,40)
  \drawline(10,5)(30,5)
  \drawline(10,5)(10,35)
  \drawline(10,35)(30,35)
  \drawline(30,5)(30,35)
  \put(32,32){\mbox{$T$}}
\end{picture}
\caption{An illustration of Lemma~\ref{lemma_cr_sec_of_reg_subsem}}\label{fig:lemma1}
\end{center}
\end{figure}
The regular submonoid $\On_n$ of $\mathcal{T}_n$ intersects each $\Lg$-class of $\mathcal{T}_n$. Indeed, every $\Lg$-class of $\mathcal{T}_n$ has the form $L_A=\{\alpha\in\mathcal{T}_n\mid \im(\alpha)=A\}$ where $A$ is a non-empty subset of $\overline{n}$, and it is easy to exhibit an order-preserving transformation in $L_A$. For instance, if $A=\{a_1,a_2,\dots,a_k\}$ with $a_1<a_2<\dots<a_k$, the following mapping does the job:
\[
\begin{pmatrix}
    \{1\}     & \{2\}    &  \dots & \{k,\dots,n\} \\
    a_1       &a_2       &  \dots & a_k \\
\end{pmatrix}.
\]
Now, Lemma~\ref{lemma_cr_sec_of_reg_subsem} shows that every $\Lg$-cross-section of $\On_n$ is an $\Lg$-cross-section of $\mathcal{T}_n$. A complete classification of $\Lg$-cross-sections of $\mathcal{T}_n$ has been obtained in \cite{Bondar2014,Bondar2016}. Hence, for a description of $\Lg$-cross-section of $\On_n$, one needs to select from the $\Lg$-cross-sections of $\mathcal{T}_n$ those which are contained in $\On_n$, that is, consist entirely of order-preserving transformation.

We will reproduce the classification from \cite{Bondar2014,Bondar2016} in Sect.~\ref{sec_L_cr_sec} below. Here we only mention that the classification shows that for each $L\in\Lg(\mathcal{T}_n)$, there is a linear order $i_1\prec i_2\prec\dots\prec i_n$ on $\overline{n}$ such that all transformations in $L$ respect this order. If $\pi$ is the permutation of $\overline{n}$ defined by $k\pi=i_k$ for $k=1,2,\dots,n$, the mapping $\alpha\mapsto\pi\alpha\pi^{-1}$ is an automorphism of $\mathcal{T}_n$. The image of $L$ under this automorphism is contained in $\On_n$ and remains an $\Lg$-cross-section of $\mathcal{T}_n$, so it is an $\Lg$-cross-section of $\mathcal{O}_n$ by Lemma~\ref{lemma_cr_sec_of_reg_subsem}. We conclude that up to isomorphism, the $\Lg$-cross-sections of $\mathcal{T}_n$ and $\On_n$ coincide.

\subsection{Dense $\R$-cross-sections}
 We start with exhibiting two series of $\R$-cross-sections of $\On_n$. The first series consists of ``dense'' cross-sections and it is inspired  by the description of $\R(\T_n)$. The second series consists of ``dual'' cross-sections and it comes from the connection between $\On_n$ and the dual of $\On_{n+1}$, combined with the description of $\Lg(\T_n)$. Then we give an example of an $\R$-cross-section of $\On_n$, which is neither dense, nor dual.

 A description of $\R$-cross-sections of $\mathcal{T}_n$ is known \cite{Pekh2003}.
For the reader's convenience, we recall  the description of $\R(\T_n)$  in what follows.
 \par Fix a linear order $\prec$ on $\overline{n}$. Let $\overline{n}=\{u_1\prec u_2\prec \ldots \prec u_n\}$. For a nonempty set $A\subset \overline{n}$ denote by $\min(A)$ the minimal element of $A$ with respect to $\prec$. If $A,B \subset \overline{n}$  are nonempty and disjoint, we will  write $A\prec B$ provided that $\min(A)\prec \min(B)$. Let $R(\prec)$ denote the set of all elements of $\mathcal{T}_n$, which have the form
$$\alpha=\begin{pmatrix}
    A_{1}       & A_{2} &  \dots & A_{k} \\
    u_{1}       & u_{2} &  \dots & u_{k} \\
   \end{pmatrix},$$
   where $A_{1}  \prec A_{2} \prec  \dots \prec A_{k}$ for all disjoint partitions $A_{1}\cup A_{2}\cup\ldots A_k$ of $\overline{n}$.

 Pekhterev has proved the following result:
  \begin{thm}[\cite{Pekh2003}] For every linear order $\prec$ on $\overline{n}$ the set $R(\prec)$ is an $\R$-cross-section of $\mathcal{T}_n$. Conversely, each $\R$-cross-section of $\mathcal{T}_n$ has the form $R(\prec)$ for some linear order $\prec$ on $\overline{n}$.
\end{thm}
 An $\R$-cross-section of $\mathcal{T}_n$, $n>2$, however,  does not belong to $\On_n$.  Indeed, the kernel classes of an order-preserving transformation are the convex subsets of the chain $\overline{n}$. That is if $x,y\in A$ for $A\in \overline{n}/\ker(\alpha)$ then $\{x\mid a<x<b\}\subset A$. But transformations of $\mathcal{T}_n$ have non-convex partitions in general. Apparently, the following fact holds.
 \begin{pro}\label{pro_dense_L} Let $(\overline{n},<)$ be a naturally ordered, $A_1, A_2,\ldots ,A_t$ be a partition of $\overline{n}$ into $t$ disjoint intervals, $1\leq t\leq n$. Then each of the sets
 \begin{gather*}
R(<)\cap\On_n=\left\{\left(
\begin{array}{cccc}
A_1 &A_2&\ldots &A_t\\
1&2&\ldots &t
\end{array}\right)\in \On_n\mid 1\leq t\leq n\right\} \mbox{ and }\\
R(<^{-1})\cap\On_n=\left\{
\left(
\begin{array}{cccc}
A_1& A_2&\ldots &A_t\\n-(t-1)&n-(t-2)&\ldots &n
\end{array}
\right)\in \On_n\mid 1\leq t\leq n\right\}
\end{gather*}
constitutes the $\R$-cross-section of $\On_n$.
\end{pro}

We call such $\R$-cross-sections to be \textbf{dense}. Hense for $n>2$ there always exist at least two (dense) $\R$-cross-sections of $\On_n$.  We will see also that the description of  $\R(\On_n)$ dose not  reduce to $\R(\T_n)\cap \On_n$ also.

\begin{example}\label{ex_dense} The dense $\R$-cross-sections of $\On_4$:
\begin{multline*}
R_1=\left\{
\left(
\begin{array}{ccc}
\{12\}&\{3\}&\{4\}\\
1&2&3
\end{array}
\right)
\left(
\begin{array}{ccc}
\{1\}&\{23\}&\{4\}\\
1&2&3
\end{array}
\right),\right.\\
\left.
\left(
\begin{array}{ccc}
\{1\}&\{2\}&\{34\}\\
1&2&3
\end{array}
\right),
\left(
\begin{array}{cc}
\{1\}&\{234\}\\
1&2
\end{array}
\right)
\left(
\begin{array}{cc}
\{12\}&\{34\}\\
1&2
\end{array}
\right),
\right.\\
\left.
\left(
\begin{array}{cccc}
\{123\}&\{4\}\\
1&2
\end{array}
\right),
\left(
\begin{array}{c}
\{1234\}\\
1
\end{array}
\right),
\left(
\begin{array}{cccc}
\{1\}&\{2\}&\{3\}&\{4\}\\
1&2&3&4
\end{array}
\right)
\right\},
\end{multline*}

\begin{multline*}
R_2=\left\{
\left(
\begin{array}{ccc}
\{12\}&\{3\}&\{4\}\\
2&3&4
\end{array}
\right)
\left(
\begin{array}{ccc}
\{1\}&\{23\}&\{4\}\\
2&3&4
\end{array}
\right),\right.\\
\left.
\left(
\begin{array}{ccc}
\{1\}&\{2\}&\{34\}\\
2&3&4
\end{array}
\right),
\left(
\begin{array}{cc}
\{1\}&\{234\}\\
3&4
\end{array}
\right)
\left(
\begin{array}{cc}
\{12\}&\{34\}\\
3&4
\end{array}
\right),
\right.\\
\left.
\left(
\begin{array}{cccc}
\{123\}&\{4\}\\
3&4
\end{array}
\right),
\left(
\begin{array}{c}
\{1234\}\\
4
\end{array}
\right),
\left(
\begin{array}{cccc}
\{1\}&\{2\}&\{3\}&\{4\}\\
1&2&3&4
\end{array}
\right)
\right\}.
\end{multline*}
\end{example}

\subsection{Dual $\R$-cross-sections of $\On_n$}\label{subsec_dual}
However, there exist other examples of $\R$-cross-sections in $\On_n$.
  Higgins~\cite{Higgins_div_On} has showed that there exists an injective homomorphism $*$ from $\On_n$ to the  dual $\On_{n+1}^*$ on $(n+1)$-element set. The  homomorphism  maps $\alpha\in \On_n$ to $\alpha^*\in\On_{n+1}^*$, where $\alpha^*$ is defined as follows.

Denote by $K=\{k_1,k_2,\ldots,k_{t}\}$  the set of the maximal members of  kernel classes of $\alpha$,  written in ascending order. Let $\im{(\alpha)}=\{r_1,r_2,\ldots, r_{t}\}$ with $k_i\alpha=r_i$ for all $1\leq i\leq t$. For each $x\in [n+1]$ define
\[
x\alpha^*=\begin{cases}
1 &\text{if } x\leq r_1,\\
k_{i}+1 &\text{if }   r_{i}< x\leq r_{i+1},\ 1\leq i< t,\\
n+1 &\text{if } x>r_t.
\end{cases}
\]
The idea is: if we take an $\Lg$-cross-section $L(\On_n)$ of $\On_n$ and consider its dual $L^*$, will we get an $\R$-cross-section of $\On_{n+1}$? Can we get all $\R$-cross-section of $\On_{n+1}$ in that way? For now we can claim the following:
\begin{pro}\label{pro_dual_L} The duals $L(\On_n)^*\cup\{\const_1\}$ and $L(\On_n)^*\cup\{\const_{n+1}\}$ are $\R$-cross-sections of $\On_{n+1}$.
\end{pro}
\begin{proof}
By the definition for $\alpha\in L(\On_n)$, $\im(\alpha)$ goes through all non-empty subsets of $\overline{n}$. The points of each subset of $\overline{n}$ breaks $[n+1]$ into convex partition. Hence, the transformations     from $L(\On_n)$ induce  $2^{n}-1$ convex partitions on $[n+1]$ (except the 1-element partition $[n+1]$ itself). Therefore the dual of $L(\On_n)$ contains exactly one representative from
  each $\R$-class of $\On_{n+1}^*$, except the constant one. Using the fact  that $^*$ is a homomorphism
  we get that $L(\On_{n})^*$ is closed under the multiplication. By the construction we have $1\alpha^*=1$ and $(n+1)\alpha^*=n+1$ for all $\alpha^*\in L(\On_n)^*$. Thus $L(\On_n)^*\cup\{\const_1\}$ and $L(\On_n)^*\cup\{\const_{n+1}\}$ constitute the $\R$-cross-sections of $\On_{n+1}$.
\end{proof}

\begin{example}
  To illustrate Proposition~\ref{pro_dual_L} we give the present example: in  Fig.~\ref{fig_dual_L} we depicted  the transformations of $\Lg$-cross-section $L$ of $\On_3$ with  solid  arrows. It is an easy exercise to verify that the set $L$ constitutes an $\Lg$-cross-section. The duals of the transformations are presented  with  dashed arrows.
\unitlength 1mm{
\begin{figure}[h]
\begin{picture}(60,75)
\gasset{AHLength=2.0,AHlength=2,AHangle=11}
\gasset{Nw=1.5,Nh=1.5,Nframe=y,Nfill=n,NLdist=3}
  \node[NLangle= 180](Na)(5,70){\small{$1$}}
  \node[NLangle= 180](Nb)(5,60){\small{$2$}}
  \node[NLangle= 180](Nc)(5,50){\small{$3$}}
  \node[NLangle= 0](N1)(15,70){\small{$1$}}
  \node[NLangle= 0](N2)(15,60){\small{$2$}}
  \node[NLangle= 0](N3)(15,50){\small{$3$}}
\drawedge(Na,N1){}
\drawedge(Nb,N2){}
\drawedge(Nc,N2){}
\gasset{AHLength=2.0,AHlength=2,AHangle=11}
\gasset{Nw=1.5,Nh=1.5,Nframe=y,Nfill=n,NLdist=3}
  \node[NLangle= 180](Na)(30,70){\small{$1$}}
  \node[NLangle= 180](Nb)(30,60){\small{$2$}}
  \node[NLangle= 180](Nc)(30,50){\small{$3$}}
  \node[NLangle= 0](N1)(40,70){\small{$1$}}
  \node[NLangle= 0](N2)(40,60){\small{$2$}}
  \node[NLangle= 0](N3)(40,50){\small{$3$}}
\drawedge(Na,N1){}
\drawedge(Nb,N3){}
\drawedge(Nc,N3){}
\gasset{AHLength=2.0,AHlength=2,AHangle=11}
\gasset{Nw=1.5,Nh=1.5,Nframe=y,Nfill=n,NLdist=3}
  \node[NLangle= 180](Na)(55,70){\small{$1$}}
  \node[NLangle= 180](Nb)(55,60){\small{$2$}}
  \node[NLangle= 180](Nc)(55,50){\small{$3$}}
  \node[NLangle= 0](N1)(65,70){\small{$1$}}
  \node[NLangle= 0](N2)(65,60){\small{$2$}}
  \node[NLangle= 0](N3)(65,50){\small{$3$}}
\drawedge(Na,N2){}
\drawedge(Nb,N3){}
\drawedge(Nc,N3){}
\gasset{AHLength=2.0,AHlength=2,AHangle=11}
\gasset{Nw=1.5,Nh=1.5,Nframe=y,Nfill=n,NLdist=3}
  \node[NLangle= 180](Na)(80,70){\small{$1$}}
  \node[NLangle= 180](Nb)(80,60){\small{$2$}}
  \node[NLangle= 180](Nc)(80,50){\small{$3$}}
  \node[NLangle= 0](N1)(90,70){\small{$1$}}
  \node[NLangle= 0](N2)(90,60){\small{$2$}}
  \node[NLangle= 0](N3)(90,50){\small{$3$}}
\drawedge(Na,N1){}
\drawedge(Nb,N1){}
\drawedge(Nc,N1){}
\gasset{AHLength=2.0,AHlength=2,AHangle=11}
\gasset{Nw=1.5,Nh=1.5,Nframe=y,Nfill=n,NLdist=3}
  \node[NLangle= 180](Na)(105,70){\small{$1$}}
  \node[NLangle= 180](Nb)(105,60){\small{$2$}}
  \node[NLangle= 180](Nc)(105,50){\small{$3$}}
  \node[NLangle= 0](N1)(115,70){\small{$1$}}
  \node[NLangle= 0](N2)(115,60){\small{$2$}}
  \node[NLangle= 0](N3)(115,50){\small{$3$}}
\drawedge(Na,N2){}
\drawedge(Nb,N2){}
\drawedge(Nc,N2){}
\gasset{AHLength=2.0,AHlength=2,AHangle=11}
\gasset{Nw=1.5,Nh=1.5,Nframe=y,Nfill=n,NLdist=3}
  \node[NLangle= 180](Na)(40,30){\small{$1$}}
  \node[NLangle= 180](Nb)(40,20){\small{$2$}}
  \node[NLangle= 180](Nc)(40,10){\small{$3$}}
  \node[NLangle= 0](N1)(50,30){\small{$1$}}
  \node[NLangle= 0](N2)(50,20){\small{$2$}}
  \node[NLangle= 0](N3)(50,10){\small{$3$}}
\drawedge(Na,N3){}
\drawedge(Nb,N3){}
\drawedge(Nc,N3){}
\gasset{AHLength=2.0,AHlength=2,AHangle=11}
\gasset{Nw=1.5,Nh=1.5,Nframe=y,Nfill=n,NLdist=3}
  \node[NLangle= 180](Na)(65,30){\small{$\textcolor{blue}{\emph{1}}$}}
  \node[NLangle= 180](Nb)(65,20){\small{\textcolor{blue}{$\emph{2}$}}}
  \node[NLangle= 180](Nc)(65,10){\small{$\textcolor{blue}{\emph{3}}$}}
  \node[NLangle= 0](N1)(75,30){\small{$\textcolor{blue}{\emph{1}
  }$}}
  \node[NLangle= 0](N2)(75,20){\small{$\textcolor{blue}{\emph{2}}$}}
  \node[NLangle= 0](N3)(75,10){\small{$\textcolor{blue}{\emph{3}}$}}
\drawedge(Na,N1){}
\drawedge(Nb,N2){}
\drawedge(Nc,N3){}
\gasset{Nw=1.5,Nh=1.5,Nframe=y,Nfill=y,NLdist=3}
  \node[NLangle= 180](Na')(5,75){\small{$\textcolor{blue}{\emph{1}}$}}
  \node[NLangle= 180](Nb')(5,65){\small{\textcolor{blue}{$\emph{2}$}}}
  \node[NLangle= 180](Nc')(5,55){\small{$\textcolor{blue}{\emph{3}}$}}
  \node[NLangle= 180](Nd')(5,45){\small{$\textcolor{blue}{\emph{4}}$}}
  \node[NLangle= 0](N1')(15,75){\small{$\textcolor{blue}{\emph{1}
  }$}}
  \node[NLangle= 0](N2')(15,65){\small{$\textcolor{blue}{\emph{2}}$}}
  \node[NLangle= 0](N3')(15,55){\small{$\textcolor{blue}{\emph{3}}$}}
  \node[NLangle= 0](N4')(15,45){\small{$\textcolor{blue}{\emph{4}}$}}
\gasset{dash={1.5}0}
\drawedge(N1',Na'){}
\drawedge(N2',Nb'){}
\drawedge(N3',Nd'){}
\drawedge(N4',Nd'){}
\gasset{Nw=1.5,Nh=1.5,Nframe=y,Nfill=y,NLdist=3}
  \node[NLangle= 180](Na')(30,75){\small{$\textcolor{blue}{\emph{1}}$}}
  \node[NLangle= 180](Nb')(30,65){\small{\textcolor{blue}{$\emph{2}$}}}
  \node[NLangle= 180](Nc')(30,55){\small{$\textcolor{blue}{\emph{3}}$}}
  \node[NLangle= 180](Nd')(30,45){\small{$\textcolor{blue}{\emph{4}}$}}
  \node[NLangle= 0](N1')(40,75){\small{$\textcolor{blue}{\emph{1}
  }$}}
  \node[NLangle= 0](N2')(40,65){\small{$\textcolor{blue}{\emph{2}}$}}
  \node[NLangle= 0](N3')(40,55){\small{$\textcolor{blue}{\emph{3}}$}}
  \node[NLangle= 0](N4')(40,45){\small{$\textcolor{blue}{\emph{4}}$}}
\gasset{dash={1.5}0}
\drawedge(N1',Na'){}
\drawedge(N2',Nb'){}
\drawedge(N3',Nb'){}
\drawedge(N4',Nd'){}
\gasset{Nw=1.5,Nh=1.5,Nframe=y,Nfill=y,NLdist=3}
  \node[NLangle= 180](Na')(55,75){\small{$\textcolor{blue}{\emph{1}}$}}
  \node[NLangle= 180](Nb')(55,65){\small{\textcolor{blue}{$\emph{2}$}}}
  \node[NLangle= 180](Nc')(55,55){\small{$\textcolor{blue}{\emph{3}}$}}
  \node[NLangle= 180](Nd')(55,45){\small{$\textcolor{blue}{\emph{4}}$}}
  \node[NLangle= 0](N1')(65,75){\small{$\textcolor{blue}{\emph{1}
  }$}}
  \node[NLangle= 0](N2')(65,65){\small{$\textcolor{blue}{\emph{2}}$}}
  \node[NLangle= 0](N3')(65,55){\small{$\textcolor{blue}{\emph{3}}$}}
  \node[NLangle= 0](N4')(65,45){\small{$\textcolor{blue}{\emph{4}}$}}
\gasset{dash={1.5}0}
\drawedge(N1',Na'){}
\drawedge(N2',Na'){}
\drawedge(N3',Nb'){}
\drawedge(N4',Nd'){}
\gasset{Nw=1.5,Nh=1.5,Nframe=y,Nfill=y,NLdist=3}
  \node[NLangle= 180](Na')(80,75){\small{$\textcolor{blue}{\emph{1}}$}}
  \node[NLangle= 180](Nb')(80,65){\small{\textcolor{blue}{$\emph{2}$}}}
  \node[NLangle= 180](Nc')(80,55){\small{$\textcolor{blue}{\emph{3}}$}}
  \node[NLangle= 180](Nd')(80,45){\small{$\textcolor{blue}{\emph{4}}$}}
  \node[NLangle= 0](N1')(90,75){\small{$\textcolor{blue}{\emph{1}
  }$}}
  \node[NLangle= 0](N2')(90,65){\small{$\textcolor{blue}{\emph{2}}$}}
  \node[NLangle= 0](N3')(90,55){\small{$\textcolor{blue}{\emph{3}}$}}
  \node[NLangle= 0](N4')(90,45){\small{$\textcolor{blue}{\emph{4}}$}}
\gasset{dash={1.5}0}
\drawedge(N1',Na'){}
\drawedge(N2',Nd'){}
\drawedge(N3',Nd'){}
\drawedge(N4',Nd'){}
\gasset{Nw=1.5,Nh=1.5,Nframe=y,Nfill=y,NLdist=3}
  \node[NLangle= 180](Na')(105,75){\small{$\textcolor{blue}{\emph{1}}$}}
  \node[NLangle= 180](Nb')(105,65){\small{\textcolor{blue}{$\emph{2}$}}}
  \node[NLangle= 180](Nc')(105,55){\small{$\textcolor{blue}{\emph{3}}$}}
  \node[NLangle= 180](Nd')(105,45){\small{$\textcolor{blue}{\emph{4}}$}}
  \node[NLangle= 0](N1')(115,75){\small{$\textcolor{blue}{\emph{1}
  }$}}
  \node[NLangle= 0](N2')(115,65){\small{$\textcolor{blue}{\emph{2}}$}}
  \node[NLangle= 0](N3')(115,55){\small{$\textcolor{blue}{\emph{3}}$}}
  \node[NLangle= 0](N4')(115,45){\small{$\textcolor{blue}{\emph{4}}$}}
\gasset{dash={1.5}0}
\drawedge(N1',Na'){}
\drawedge(N2',Na'){}
\drawedge(N3',Nd'){}
\drawedge(N4',Nd'){}
\gasset{Nw=1.5,Nh=1.5,Nframe=y,Nfill=y,NLdist=3}
  \node[NLangle= 180](Na')(40,35){\small{$\textcolor{blue}{\emph{1}}$}}
  \node[NLangle= 180](Nb')(40,25){\small{\textcolor{blue}{$\emph{2}$}}}
  \node[NLangle= 180](Nc')(40,15){\small{$\textcolor{blue}{\emph{3}}$}}
  \node[NLangle= 180](Nd')(40,5){\small{$\textcolor{blue}{\emph{4}}$}}
  \node[NLangle= 0](N1')(50,35){\small{$\textcolor{blue}{\emph{1}
  }$}}
  \node[NLangle= 0](N2')(50,25){\small{$\textcolor{blue}{\emph{2}}$}}
  \node[NLangle= 0](N3')(50,15){\small{$\textcolor{blue}{\emph{3}}$}}
  \node[NLangle= 0](N4')(50,5){\small{$\textcolor{blue}{\emph{4}}$}}
\gasset{dash={1.5}0}
\drawedge(N1',Na'){}
\drawedge(N2',Na'){}
\drawedge(N3',Na'){}
\drawedge(N4',Nd'){}
\gasset{Nw=1.5,Nh=1.5,Nframe=y,Nfill=y,NLdist=3}
  \node[NLangle= 180](Na')(65,35){\small{$\textcolor{blue}{\emph{1}}$}}
  \node[NLangle= 180](Nb')(65,25){\small{\textcolor{blue}{$\emph{2}$}}}
  \node[NLangle= 180](Nc')(65,15){\small{$\textcolor{blue}{\emph{3}}$}}
  \node[NLangle= 180](Nd')(65,5){\small{$\textcolor{blue}{\emph{4}}$}}
  \node[NLangle= 0](N1')(75,35){\small{$\textcolor{blue}{\emph{1}
  }$}}
  \node[NLangle= 0](N2')(75,25){\small{$\textcolor{blue}{\emph{2}}$}}
  \node[NLangle= 0](N3')(75,15){\small{$\textcolor{blue}{\emph{3}}$}}
  \node[NLangle= 0](N4')(75,5){\small{$\textcolor{blue}{\emph{4}}$}}
\gasset{dash={1.5}0}
\drawedge(N1',Na'){}
\drawedge(N2',Nb'){}
\drawedge(N3',Nc'){}
\drawedge(N4',Nd'){}
\end{picture}
\caption{$\Lg$-cross-section of $\mathcal{T}_3$ and its dual in $\mathcal{T}_{4}$}\label{fig_dual_L}
\end{figure}
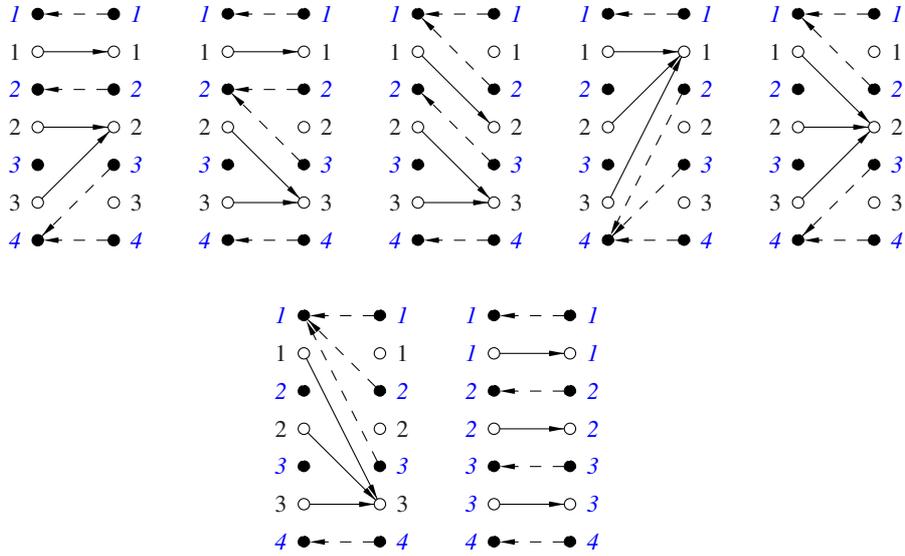}

Consider the the set $L^*$ of the dual transformations of $L$:

\begin{gather*}
L^*=\left\{
\left(
\begin{array}{ccc}
\{1\}&\{2\}&\{34\}\\
1    &2    &4
\end{array}
\right)
\left(
\begin{array}{ccc}
\{1\}&\{23\}&\{4\}\\
1     &2    &4
\end{array}
\right),
\left(
\begin{array}{ccc}
\{12\}&\{3\}&\{4\}\\
1     &2&4
\end{array}
\right)
\right.,\\
\left.
\left(
\begin{array}{cc}
\{1\}&\{234\}\\
1   &4
\end{array}
\right),
\left(
\begin{array}{cc}
\{12\}&\{34\}\\
1&4
\end{array}
\right),
\left(
\begin{array}{cc}
\{123\}&\{4\}\\
1       &4
\end{array}
\right),
\left(
\begin{array}{cccc}
\{1\}&\{2\}&\{3\}&\{4\}\\
1&2&3&4
\end{array}
\right)
\right\}.
\end{gather*}
It is easy to verify that  each of the sets $L^*\cup \{\const_1\}$ or $L^*\cup \{\const_4\}$ forms an $\R$-cross-section of $\On_{4}$. The $\R$-cross-sections of $\On_{n}$ obtained in this way we will call  \textbf{dual} for brevity.
\end{example}

The direct calculations show that the dual  preimages of $R_1\setminus \const_1$ and $R_2\setminus \const_4$ of dense cross-sections $R_1$ and $R_2$ do not form the $\Lg$-cross-sections of $\On_3$.


\begin{definition}
We say that an  $\R$-cross-section $R$ \textbf{has a fixed point} $x$, if $xf=x$ for all non-constant $f\in R$.
\end{definition}
Note that each of dense $\R$-cross-sections of $([n+1],\prec)$ always has the only fixed point: either 1, or $n+1$. On the other hand, by the construction each set $L(\On_n)^*$ always has at least  two fixed  points, namely, 1  and $n+1$.

\begin{example}\label{ex_neither_nor} However, there exists an $\R$-cross-section of $\On_4$ which is not dense but has the only fixed point 1:
\begin{gather*}
R_5=\left\{
\left(
\begin{array}{ccc}
\{1\}&\{2\}&\{34\}\\
1&2&3
\end{array}
\right),
\left(
\begin{array}{ccc}
\{1\}&\{23\}&\{4\}\\
1&3&4
\end{array}
\right),\right.\\
\left.
\left(
\begin{array}{ccc}
\{12\}&\{3\}&\{4\}\\
1&3&4
\end{array}
\right)
\left(
\begin{array}{cc}
1&\{234\}\\
1&3
\end{array}
\right),
\left(
\begin{array}{cc}
\{12\}&\{34\}\\
1&3
\end{array}
\right),\right.\\
\left.
\left(
\begin{array}{cc}
\{123\}&4\\
1&3
\end{array}
\right),
\left(
\begin{array}{c}
\{1234\}\\
1
\end{array}
\right),
\left(
\begin{array}{cccc}
\{1\}&\{2\}&\{3\}&\{4\}\\
1&2&3&4
\end{array}
\right)
\right\}.
\end{gather*}
\end{example}

Thus, at the moment we have examples of two types of $\R$-cross-sections of $\On_n$ with completely different origins. We have also the evidence of the existing of other ones. In fact, we will show that all $\R$-cross-sections of $\On_n$  can be generated in some sense by a sequence of  $\Lg$-cross-sections of $\On_{i_1},\On_{i_2},\ldots\On_{i_k}$, with $i_1+i_2+\ldots i_k=n$. The sequence in turn is tightly connected with a certain type of binary search trees and their duals, so-called inner trees. We will  state the necessary types of binary trees, examples and   interpretations trees as diagrams in the following section.

\section{Trees and diagrams}\label{sec_trees_and_diagrams}
  Recall that \textit{a rooted binary tree} is an  acyclic connected graph in which one vertex is specified as a root, each vertex $v$  has at most two children  and  a unique parent (except the root, which has no one).  The trees occurring in this paper assumed to be rooted binary trees unless otherwise stated. We denote the root of a binary tree by $\rt$, the parent of a vertex $v$ by $\p(v)$. The children are referred to as the left child (the son of a vertex) and the right child (the daughter of a vertex). The son and the daughter of $v$ are denoted by $\s(v)$ and $\dau(v)$, respectively. A leaf is a vertex of the tree which has  no children. Vertices which are not leaves are called internal nodes. Each internal vertex $v$ of a tree has the left and right  subtree.  A subtree may be empty.  The non-empty subtree consists of a vertex (a child of $v$) and  the descendants of the vertex in the tree.  A binary tree is said to be full if each its internal vertex has exactly two children.

  \par There exists  a unique path from the root to a vertex in the tree. We denote by $\omega(v)$ the set $\{v,\p(v),\p(\p(v)),\ldots,\rt\}$ containing the path of a  vertex  $v$. The level of a vertex is the length of its path. Hence, the level of the root equals $0$; the level of a non-root vertex is the level of its parent plus 1.

    \par Let $T(n)$ denote a tree whose vertices are labeled with $1,2,\ldots n$. We do not make a difference between  a vertex and its label for convenience. So, if a vertex $v$ is labeled with  a number $a$ with  $a< b$ for some $b\in \overline{n}$ we write $v<b$.

%

\subsection{Order-preserving binary tree}

 \par In computer science, a binary search tree (BST) \cite{Cormen}  is a rooted binary tree with the following property: for a vertices $x, y$ of the tree if $y$ is a vertex in the left subtree of $x$ then $y\leq x$. If $y$ is a vertex in the right subtree of $x$ then $y\geq x$.
 We need to modify the notion of a BST for our purposes.

 \begin{definition}\label{def_ord_pres}
\textbf{An order-preserving tree} is a rooted binary tree $T(n)$  with the following property:  for a vertices $x, y\in T(n)$ if $y$ is a vertex in the left subtree of $x$ then $y<x$. If $y$ is a vertex in the right subtree of $x$ then $x<y$.
\end{definition}
Thus, for a natural $n$, an order-preserving binary tree is a kind of strict BST with the vertex set $\{1,2,\ldots, n\}$.
It is also important for us to fix  the bounds of each vertex in the tree:
\begin{definition}\label{def_can_b} We define  the \textbf{ canonical bounds} of a vertex of an order-preserving binary tree  $T(n)$  by induction:

1)  For the root the canonical bounds are  $1\leq\rt\leq n$.

2) Let  $v\in T(n)$ be a  vertex   with the canonical bounds $a\leq v\leq b$ which has been already defined. If $v$ has a son (a daughter)  then a child has the following canonical bounds
 $$a\leq\s(v)<v\mbox{ \ and \ }v<\dau(v)\leq b \mbox{ \ respectively}.$$
\end{definition}

\begin{figure}[h]
\begin{center}
 \unitlength=0.50mm
\begin{picture}(135,70)
\gasset{AHnb=0}
\gasset{Nw=10,Nh=10,Nmr=5}
\put(0,70){$T(4)$}
\node(1)(5,60){1}
\node(2)(15,45){2}
\node(4)(25,30){4}
\node(6)(15,15){3}
\drawedge(1,2){}
\drawedge(2,4){}
\drawedge(4,6){}
\put(55,70){$T_1(5)$}
\node(t1)(60,60){2}
\node(t2)(50,45){1}
\node(t3)(70,45){4}
\node(t4)(60,30){3}
\node(t5)(80,30){5}
\drawedge(t1,t2){}\drawedge(t1,t3){}
\drawedge(t3,t4){}\drawedge(t3,t5){}
\put(130,70){$T_2(5)$}
\node(s1)(135,60){5}
\node(s2)(120,45){1}
\node(s4)(135,30){4}
\node(s5)(120,15){3}
\node(s6)(105,0){2}
\drawedge(s1,s2){}
\drawedge(s2,s4){}\drawedge(s4,s5){}
\drawedge(s5,s6){}
\end{picture}
\caption{Examples of order-preserving binary trees for $n=4,5$. }\label{fig:ord_pres_derevya}
\end{center}
\end{figure}
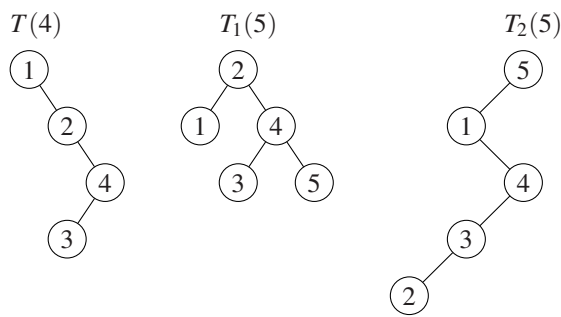
The definition implies that each canonical bound of a vertex $v$ is either belongs to $\omega(v)$, or to the set $\{1,n\}$.

Note that an order-preserving binary tree on $T(n)$ can be regarded as a model $(\overline{n},\leq,\prec)$, where $\prec$ is a strong partial order on  $\overline{n}$, defined as follows: for all $x,y\in \overline{n}$ set $x\prec y$ if $x$ is a descendant of $y$. Since $\overline{n}$ assumed to be equipped with the natural order $\leq$ through the paper, we write $(\overline{n},\prec)$ instead $(\overline{n},\leq,\prec)$, or just $T(n)$.

\par  It is convenient for our purposes to picture an order-preserving binary tree  of $(\overline{n},\prec)$ in an "unfolded"\  form (diagram) which is described below. The main advantage of the diagrams is that it ``preserves the scale'' of the tree and makes the ``inner'' trees (which we define further) more visual.

By $U_q$, $0\leq q\leq n-1$, we denote the set of all vertices in the $q$-th level of an order-preserving tree $T(n)$.  To each  vertex $x\in U_q$ of the tree we assign the point $(x,q)$ and the line segment with the endpoints $(x,q)$ and $(x,n-1)$ of the coordinate plane (see Fig.~\ref{fig:primer diagrammi}).
\par Since $T(n)$ is order-preserving, the children of each $(x,q)$ are determined by the diagram of the tree in a unique way. The points of the same level are incomparable with respect to $\prec$.

\begin{figure}[h]
\begin{center}
\begin{picture}(40,50)
\gasset{AHnb=0,linewidth=0.4}
  \drawline[linewidth=0.1](0,10)(40,10)
  \drawline[linewidth=0.1](0,15)(40,15)
  \drawline[linewidth=0.1](0,20)(40,20)
  \drawline(0,25)(5,25) \drawline[linewidth=0.1](5,25)(15,25) \drawline(15,25)(20,25)\drawline[linewidth=0.1](20,25)(35,25)\drawline(35,25)(40,25)
  \drawline(0,30)(20,30)\drawline[linewidth=0.1](20,30)(30,30)\drawline(30,30)(40,30)
  \drawline[linewidth=0.1](0,35)(10,35)\drawline(10,35)(30,35)\drawline[linewidth=0.1](30,35)(40,35)
  \drawline[linewidth=0.1](0,40)(40,40)
 \put(-8,40){$U_0$}
  \put(-8,35){$U_1$}
  \put(-8,30){$U_2$}
  \put(-10,22){$\ldots$}
  \put(-15,15){$U_{n-2}$}
  \put(-15,10){$U_{n-1}$}
  \put(-1,43){$1$}
  \put(4,43){$2$}
  \put(13,43){$\ldots$}
  \put(24,43){$r$}
  \put(30,43){$\ldots$}
  \put(39,43){$n$}
  \drawline(0,10)(0,40)
  \drawline(5,10)(5,40)
  \drawline(10,10)(10,40)
  \drawline(15,10)(15,40)
  \drawline(20,10)(20,40)
  \drawline(25,10)(25,40)
  \drawline(30,10)(30,40)
  \drawline(35,10)(35,40)
  \drawline(40,10)(40,40)
  \gasset{AHnb=0,linewidth=0.8, Nw=1.5,Nh=1.5,Nframe=y, Nfill=y, NLdist=3}
  \node(r)(25,40){}
  \node[NLangle= 45](r_21)(10,35){}
  \node[NLangle= 45](r_22)(30,35){}
  \node(r_21)(0,30){}
  \node(r_22)(20,30){}
  \node(r_23)(40,30){}
  \node(r_31)(5,25){}
  \node(r_32)(15,25){}
  \node(r_33)(35,25){}
    \drawline(25,10)(25,40)
    \drawline(10,10)(10,35)
    \drawline(30,10)(30,35)
     \drawline(0,10)(0,30)
    \drawline(20,10)(20,30)
    \drawline(40,10)(40,30)
    \drawline(5,10)(5,25)
    \drawline(15,10)(15,25)
    \drawline(35,10)(35,25)
  \end{picture}
\caption{"Unfolded"\ representation (diagram) of the order-preserving binary tree  //$(\overline{n},\leq,\prec)$.}\label{fig:primer diagrammi}
\end{center}
\end{figure}
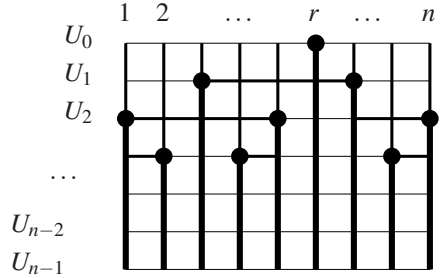
Note that the diagram  of an order-preserving binary tree completely determines the original tree.
For example, the diagrams of the trees  from Fig.~\ref{fig:ord_pres_derevya} are presented in Fig.~\ref{fig:diagramms}.

\begin{figure}[ht]
\begin{center}
\unitlength0.8mm
\begin{picture}(90,35)
\gasset{AHnb=0,linewidth=0.4}
\put(5,0){$T(4)$}
  \drawline[linewidth=0.1](0,10)(10,10)\drawline(10,10)(15,10)
  \drawline[linewidth=0.1](0,15)(5,15)\drawline(5,15)(15,15)
  \drawline(0,20)(5,20)\drawline[linewidth=0.1](5,20)(15,20)
  \drawline[linewidth=0.1](0,25)(15,25)
  \drawline(0,10)(0,25)
  \drawline(5,10)(5,25)
  \drawline(10,10)(10,25)
  \drawline(15,10)(15,25)
 \put(-5,24){$0$}
  \put(-5,19){$1$}
  \put(-5,15){$2$}
  \put(-5,9){$3$}
  \put(-1,27){$1$}
  \put(4,27){$2$}
  \put(9,27){$3$}
  \put(14,27){$4$}
\put(35,0){$T_1(5)$}
  \drawline[linewidth=0.1](30,10)(50,10)
  \drawline[linewidth=0.1](30,15)(50,15)
  \drawline[linewidth=0.1](30,20)(40,20)\drawline(40,20)(50,20)
  \drawline(30,25)(45,25)\drawline[linewidth=0.1](45,25)(50,25)
  \drawline[linewidth=0.1](30,30)(50,30)
  \drawline(30,10)(30,30)
  \drawline(35,10)(35,30)
  \drawline(40,10)(40,30)
  \drawline(45,10)(45,30)
  \drawline(50,10)(50,30)
  \put(25,29){$0$}
 \put(25,24){$1$}
  \put(25,19){$2$}
  \put(25,15){$3$}
  \put(25,9){$4$}
  \put(29,33){$1$}
  \put(34,33){$2$}
  \put(39,33){$3$}
  \put(44,33){$4$}
  \put(49,33){$5$}
\put(75,0){$T_2(5)$}
  \drawline[linewidth=0.1](65,10)(70,10)\drawline(70,10)(75,10)\drawline[linewidth=0.1](75,10)(85,10)
  \drawline[linewidth=0.1](65,15)(75,15)\drawline(75,15)(80,15)\drawline[linewidth=0.1](80,15)(85,15)
  \drawline(65,20)(80,20)\drawline[linewidth=0.1](80,20)(85,20)
  \drawline(65,25)(85,25)
  \drawline[linewidth=0.1](65,30)(85,30)
   \drawline(65,10)(65,30)
  \drawline(70,10)(70,30)
  \drawline(75,10)(75,30)
  \drawline(80,10)(80,30)
  \drawline(85,10)(85,30)
  \put(60,29){$0$}
 \put(60,24){$1$}
  \put(60,19){$2$}
  \put(60,15){$3$}
  \put(60,9){$4$}
  \put(64,33){$1$}
  \put(69,33){$2$}
  \put(74,33){$3$}
  \put(79,33){$4$}
  \put(84,33){$5$}
\gasset{AHnb=0,linewidth=0.8, Nw=1.5,Nh=1.5,Nframe=y, Nfill=y, NLdist=3}
  \node(r)(0,25){}
  \node(r_21)(5,20){}
  \node(r_22)(10,10){}
  \node(r_21)(15,15){}
  \drawline(0,25)(0,10)
  \drawline(5,20)(5,10)
  \drawline(15,15)(15,10)
  \node(r_22)(30,25){}\drawline(30,25)(30,10)
  \node(r_23)(35,30){}\drawline(35,30)(35,10)
  \node(r_31)(40,20){}\drawline(40,20)(40,10)
  \node(r_32)(45,25){}\drawline(45,25)(45,10)
  \node(r_33)(50,20){}\drawline(50,20)(50,10)
  \node(r_22)(85,30){}\drawline(85,30)(85,10)
  \node(r_23)(65,25){}\drawline(65,10)(65,25)
  \node(r_33)(80,20){}
  \node(r_31)(75,15){}\drawline(75,15)(75,10)
  \node(r_32)(70,10){}\drawline(80,20)(80,10)
  \drawline(85,25)(85,10)
  \end{picture}
\caption{Diagrams of the order-preserving binary trees shown on Fig.\ref{fig:ord_pres_derevya}.}\label{fig:diagramms}
\end{center}
\end{figure}
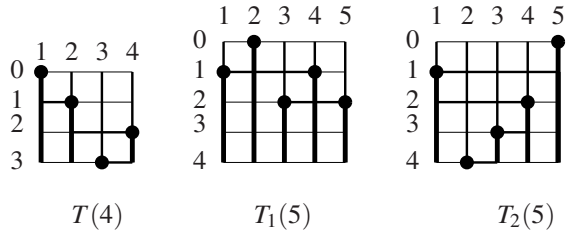
\par Everywhere below the maximal element in an interval with respect to $\prec$  we call the \textbf{highest}, while the word ``maximal'' stands for the maximal element with respect to $\leq$. We also say the element is the \textbf{lowest} if it is minimal with respect to $\prec$, while the word ``minimal'' relates with usual order.

\par Our goal now to give the notion of a decreasing tree, which is  a special case of order-preserving trees. This leads us to the notion of so-called "inner"  binary tree of an order-preserving tree, whose vertices are marked by intervals. Inner trees can be extremely well seen  in the diagram of the order-preserving tree. We discuss this construction more detail in the next subsection.

\subsection{Inner binary tree}

 Let $i,j\in \overline{n}$ and $i\leq j$. The interval $[i,j]$ is the set $\{k\in \overline{n}\mid i\leq k\leq j\}$. We write $[i]$ instead of $[i,i]$.
\par Consider the diagram of an order-preserving tree $T(n)$.  Label the cells in diagrams  from left to right in following way: the cell which is next after vertex $i$, $1\leq i< n$, we label by $i'$.
\begin{definition}
We define the \textbf{inner tree} $\Gamma$ of an order-preserving tree $T(n)$ as a full binary tree  labeled with intervals of $[1',(n-1)']$, defined by induction as follows:
\par 1. The interval $[1',(n-1)']$ is the root of $\Gamma$.
\par 2. Let $[i',j']\in\Gamma$. If $i'= j'$ then $[i']$ is a leaf. Let $i'\ne j'$, $x\in T(n)$ be the highest point with $i<x<j+1$. Then $[i',(x-1)']$ and $[x',j']$ are the daughter and the son of $[i',j']$ respectively.
 \end{definition}
 \begin{example}\label{ex_inner_tree}
According to the previous definition the inner tree of $T(5)$ has the form (see Fig.~\ref{fig:derevya_otrezkov}):
\begin{figure}[h]
\begin{center}
\unitlength=1.05mm
\begin{picture}(60,35)
\gasset{AHnb=0, Nw=9,Nh=9}
\node(t1)(50,35){$[1',4']$}
\node[NLangle= 150](t2)(45,25){$[1']$}
\node[NLangle= 30](t3)(55,25){$[2',4']$}
\drawedge(t1,t2){}\drawedge(t1,t3){}
\node[NLangle= 30](t4)(50,15){$[2',3']$}
\node[NLangle= 30](t5)(60,15){$[4']$}
\drawedge(t3,t4){}\drawedge(t3,t5){}
\node[NLangle= 30](t6)(45,5){$[2']$}
\node[NLangle= 30](t7)(55,5){$[3']$}
\drawedge(t6,t4){}\drawedge(t7,t4){}
\put(60,35){$\Gamma$}
\gasset{AHnb=0,linewidth=0.4}
\put(5,5){$T(5)$}
  \drawline[linewidth=0.1](0,10)(20,10)
  \drawline[linewidth=0.1](0,15)(20,15)\drawline(10,15)(15,15)
  \drawline[linewidth=0.1](0,20)(15,20)\drawline(15,20)(20,20)
  \drawline(0,25)(20,25)\drawline[linewidth=0.1](15,25)(20,25)
  \drawline[linewidth=0.1](0,30)(20,30)
  \drawline(0,10)(0,30)
  \drawline(5,10)(5,30)
  \drawline(10,10)(10,30)
  \drawline(15,10)(15,30)
  \drawline(20,10)(20,30)
  \put(-5,29){$0$}
 \put(-5,24){$1$}
  \put(-5,19){$2$}
  \put(-5,15){$3$}
  \put(-5,9){$4$}
  \put(-1,33){$1$}
  \put(4,33){$2$}
  \put(9,33){$3$}
  \put(14,33){$4$}
  \put(19,33){$5$}
  \put(2,26){$1'$}
  \put(7,26){$2'$}
  \put(12,26){$3'$}
  \put(17,26){$4'$}
\gasset{AHnb=0,linewidth=0.8, Nw=1.5,Nh=1.5,Nframe=y, Nfill=y, NLdist=3}
  \node(r_22)(0,25){}\drawline(0,25)(0,10)
  \node(r_23)(5,30){}\drawline(5,30)(5,10)
  \node(r_32)(20,25){}\drawline(20,25)(20,10)
  \node(r_33)(10,15){}\drawline(10,15)(10,10)
  \node(r_33)(15,20){}\drawline(15,20)(15,10)
\end{picture}
\caption{The inner tree $\Gamma$ of the order-preserving tree $T(5)$.}\label{fig:derevya_otrezkov}
\end{center}
\end{figure}
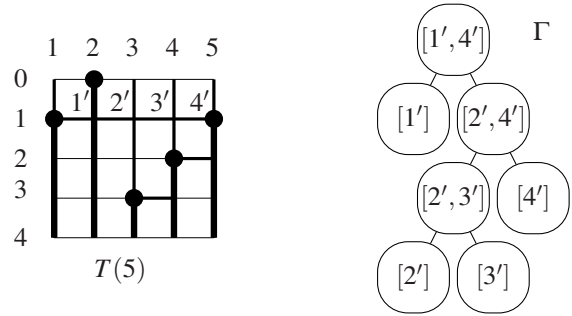

\par  Note that the inner tree is determined by an order-preserving tree in a unique way. Nevertheless the same inner tree could be associated with different order-preserving trees since an order-preserving tree is not full in general. For instance, $\Gamma$ is also the inner tree of $T_1(5)$ (see Fig.~\ref{fig:diagramms}).
\end{example}

Let $x\in T(n)$ be a vertex with the canonical bounds $a,b$, $a<b$. Then \textbf{the left inner tree} $\Gamma_l(x)$ and \textbf{the right inner tree} $\Gamma_r(x)$ of $x$ are the subtrees of the inner tree $\Gamma$ rooted at $[a',(x-1)']$ and $[x',(b- 1)']$ respectively. We define  $\Gamma_l(x)=\emptyset$ if $x=1$ and $\Gamma_r(x)=\emptyset$ if $x=n$.

\begin{example}
Consider tree $T(4)$ from Fig.~\ref{fig:diagramms}. For the vertex $2\in T(4)$ the left inner tree $\Gamma_l(2)$ is isomorphic to a point (1-element full binary tree), tree $\Gamma_r(2)$ is isomorphic to a 2-element full binary tree (see Fig.~\ref{fig:left_right_inner tree}). The right inner tree $\Gamma_r(4)$ of $4$ is empty.
\begin{figure}[h]
\begin{center}
\unitlength=1.05mm
\begin{picture}(60,25)
\gasset{AHnb=0, Nw=9,Nh=9}
\put(-5,18){$\Gamma_l(2)$}
\node(1)(10,9){$[1']$}
\put(40,22){$\Gamma_r(2)$}
\node(t1)(45,15){$[2',3']$}
\node[NLangle= 150](t2)(40,5){$[2']$}
\node[NLangle= 30](t3)(50,5){$[3']$}
\drawedge(t1,t2){}\drawedge(t1,t3){}
\end{picture}
\caption{The inner trees $\Gamma_l(2)$ and $\Gamma_r(2)$}\label{fig:left_right_inner tree}
\end{center}
\end{figure}
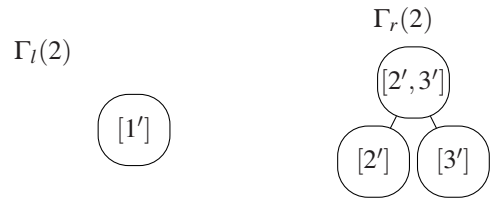
\end{example}
Exactly the interconnections between  inner trees of the vertices play a crucial role in the definition of decreasing binary trees.
\subsection{Decreasing binary tree}
 We will need ``to compare'' the trees frequently in the sense of following definitions.
 \begin{definition}
 By a \emph{homomorphism} between two trees $T_1$ and $T_2$ we mean a 1-1-map from the vertex
 set of $T_1$ into the vertex set of $T_2$ that sends the root of $T_1$ to the root of $T_2$ and preserves the parent--child relation and the genders of non-root vertices.
 \end{definition}

 \begin{definition} Given two trees $T_1$ and $T_2$, we say that $T_1$ \emph{subordinates} $T_2$
 (and write $T_1\hookrightarrow T_2$) if there exists a homomorphism $T_1\to T_2$ \cite{Bondar_Volkov_CRA}.
  The empty tree is assumed to subordinate every binary tree.
  \end{definition}

\par Recall that by $\omega(1)$, $\omega(n)$ we denote the paths from the vertex $1$ and $n$ respectively, to the root in  $(\overline{n},\prec)$.
\begin{definition}\label{def_ord_pres} We say an order-preserving tree  $(\overline{n},\prec)$ to be \textbf{decreasing}, if for all  $x, y\ne\rt\in (\overline{n},\prec)$ with $x\prec y$ the following conditions hold:
  \begin{enumerate}

 \item  if $x, y\in\omega(1)$ then  the right inner tree  of $x$  subordinates the right inner tree  of $y$;
 \item if $x, y\in\omega(n)$ then the left inner tree  of $x$  subordinates the left inner tree  of $y$;
  \item if $x, y\notin\omega(1)\cup\omega(n)$ then the inner trees of $x$ subordinate the respective inner trees of $y$.
  \end{enumerate}
\end{definition}
\begin{example} The tree $T(4)$ from Fig.~\ref{fig:diagramms} is not decreasing, since $4\prec 2$ and $2,4\in\omega(4)$, but it does not hold $\Gamma_l(4)\hookrightarrow\Gamma_l(2)$.
 Trees $T_1(5)$ and $T_2(5)$ are decreasing.
\end{example}

\subsection{Order-preserving tree of a convex partition of $\overline{n}$.}
It remains to define one more type of binary trees. The one is defined for a fixed order-preserving tree $(\overline{n},\prec)$  and a convex partition of $\overline{n}$.

 Let $\widetilde{K}=K_1\cup K_2\cup\ldots\cup K_m$ be a convex partition of $\overline{n}$ into $m$ intervals:

$$K_1=[k_0=1, k_1] \mbox{ and } K_i=[k_{i-1}+1, k_i], \mbox{ for } 2\leq i\leq m.$$

For every $a\in \overline{n}$  denote by $K^{(a)}$  the interval of $\widetilde{K}$ that contains $a$.

\begin{definition}\label{def_part_tree} Define a partition  order-preserving tree $T(\widetilde{K})$  with respect to $(\overline{n},\prec)$ by induction as follows.

1. Let $K^{(\rt)}\in\widetilde{K}$ be the root  of $T(\widetilde{K})$.

2. Assume $V$ is a vertex of $T(\widetilde{K})$. By $m_s\in \overline{n}$ we denote the highest and the  closest to $V$ on the left vertex of $(\overline{n},\prec)$, which does not belong to any of already defined vertices of $T(\widetilde{K})$. In dual way, by $m_d$ denote the highest and the  closest vertex to $V$ on the right. Then
 $$\s(V):=K^{(m_s)} \mbox{\ \ and \ }\dau(V):=K^{(m_d)}.$$
\end{definition}
For  brevity in what follows we refer to $T(\widetilde{K})$ as to the partition tree. The vertices $\rt$, $m_s$ ($m_d$) are said to be
the \textbf{leading} element of the interval.
\begin{example}\label{ex_partition_tree} Let  $T_3(5)$ be the decreasing tree (see Fig.\ref{fig:ex_part_tree}),  $\widetilde{K}=\{\{1,2\},\{3,4\},\{5\}\}$. Since $\rt=3$, we have $K^{(\rt)}=\{3,4\}$ is a root of $T(\widetilde{K})$. The leading elements of $\{1,2\}$ and $\{5\}$ are $m_s=1$, $m_d=5$ respectively. Thus $\s(\{3,4\})=\{1,2\}$,  $\dau(\{3,4\})=\{5\}$.

\begin{figure}[h]
\begin{center}
\unitlength0.8mm
\begin{picture}(100,35)
\gasset{AHnb=0,Nw=10.5,Nh=10.5, ,Nframe=y, Nfill=n}
\node(ti0)(95,25){$\{3,4\}$}
\node[NLangle= 150](ti1)(90,12){$\{1,2\}$}
\node[NLangle= 30](ti2)(100,12){$\{5\}$}
\drawedge(ti2,ti0){}\drawedge(ti1,ti0){}
\put(90,0){$T(\widetilde{K})$}
\gasset{AHnb=0,linewidth=0.4}
  \drawline[linewidth=0.1](0,10)(20,10)
  \drawline[linewidth=0.1](0,15)(20,15)
  \drawline(0,20)(5,20)\drawline[linewidth=0.1](5,20)(15,20)\drawline(15,20)(20,20)
  \drawline(0,25)(15,25) \drawline[linewidth=0.1](15,25)(20,25)
  \drawline[linewidth=0.1](0,30)(20,30)
  \drawline(0,10)(0,30)
  \drawline(5,10)(5,30)
  \drawline(10,10)(10,30)
  \drawline(15,10)(15,30)
  \drawline(20,10)(20,30)
  \put(-5,29){$0$}
 \put(-5,24){$1$}
  \put(-5,19){$2$}
  \put(-5,14){$3$}
  \put(-5,9){$4$}
  \put(-1,33){$1$}
  \put(4,33){$2$}
  \put(9,33){$3$}
  \put(14,33){$4$}
  \put(19,33){$5$}
\gasset{AHnb=0,linewidth=0.8, Nw=1.5,Nh=1.5,Nframe=y, Nfill=y, NLdist=3}
  \node(r_22)(0,25){}\drawline(0,25)(0,10)
  \node(r_23)(5,20){}\drawline(5,20)(5,10)
  \node(r_31)(10,30){}\drawline(10,30)(10,10)
  \node(r_32)(15,25){}\drawline(15,25)(15,10)
  \node(r_33)(20,20){}\drawline(20,20)(20,10)
  \put(3,0){$T_3(5)$}
\end{picture}
\caption{The partition tree $T(\widetilde{K})$ for $T_3(5)$, $\widetilde{K}=\{\{1,2\},\{3,4\},\{5\}\}$.}\label{fig:ex_part_tree}
\end{center}
\end{figure}
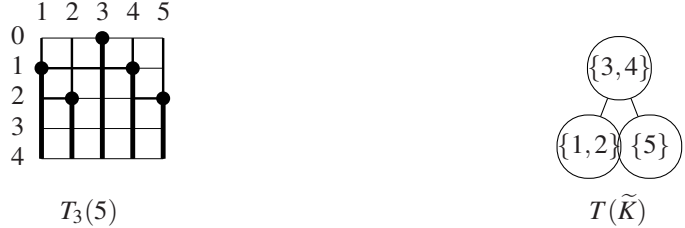
\end{example}

\section{Description of $\R$-cross-sections of $\On_n$}
In this section we introduce a semigroup $\Phi_\prec\subset \On_n$. We will prove that the semigroup $\Phi_\prec$ forms an $\R$-cross-section and conversely, each $\R$-cross-section of $\On_n$ is given by $\Phi_{\prec}$ for a certain $\prec$.

\subsection{Semigroup $\Phi_\prec$}
 Let $(\overline{n},\prec)$ be a decreasing  tree. Fix   a convex partition $\widetilde{K}$  of~$\overline{n}$ into $m$ intervals. Let  $T(\widetilde{K})$ be the partition tree associated with $(\overline{n},\prec)$. First we show that there exists a unique homomorphism from $ T(\widetilde{K})$ to $(\overline{n},\prec)$.
 \par Note that the subtree of $(\overline{n},\prec)$, whose vertex set is $\omega(1)$, contains only the  male  vertices (except the root).  The subtree of $(\overline{n},\prec)$, whose vertex set is $\omega(n)$, contains only the  female  vertices (except the root). We denote also by $\omega(K^{(1)})$ and $\omega(K^{(n)})$ the  subtrees of $T(\widetilde{K})$ whose vertex sets are the paths of $K^{(1)}$ and $K^{(n)}$ respectively.
 \begin{lem}\label{lemma leadings of omega}
For all $K\in\omega(K^{(1)})$ the leading element of $K$ belongs to $\omega(1)$. Dually, for all $K\in \omega(K^{(n)})$ the leading element of $K$ belongs to $\omega(n)$.
 \end{lem}
 \begin{proof} Without loss of generality we consider $\omega(K^{(1)})$. We induct on the level of a vertex $K\in\omega(K^{(1)})$. The root of $\omega(K^{(1)})$ always contains $\rt\in\omega(1)$. Therefore, the  induction base holds. Suppose $K\in\omega(K^{(1)})$ and the assumption holds: the leading element of $K$ is $a\in\omega(1)$.  If $1\in K$ there is nothing to prove. Suppose $1\notin K$. Then $a>1$. Let $x$ be the minimal element in $K$. Assume $x\notin \omega(1)$. By the construction of the order-preserving tree there exists  a unique minimal ancestor $y\in\omega(1)$ with $y<x$. Clearly, other  ancestors on the left are lower than $y$.  Thus, $y$ is the leading element of $\s(K)$. If $x\in \omega(1)$ then $\s(x)\in\omega(1)$ is a unique vertex satisfying Definition~\ref{def_part_tree}.

 The proof for $\omega(K^{(n)})$ is dual.
 \end{proof}
 \begin{lem}\label{lemma_path_to_path}
 If there exists a  homomorphism $f\colon T(\widetilde{K})\to (\overline{n},\prec)$ then
 $ \omega(K^{(1)})\stackrel{f}{\hookrightarrow}  \omega(1)$ and $ \omega(K^{(n)})\stackrel{f}{\hookrightarrow}  \omega(n)$. Moreover,
   if $K\notin\omega(K^{(1)})$ then $Kf\notin\omega(1)$;  if $K\notin\omega(K^{(n)})$ then $Kf\notin\omega(n)$.
 \end{lem}
 \begin{proof}
 Without loss of generality we assume $\rt\ne 1$ and consider  the left subtrees $\omega(1)$ and $\omega(K^{(1)})$. It is easy to see that  $|\omega(K^{(1)})|\leq|\omega(1)|$ by  Definition~\ref{def_part_tree}.  We induct on the level  of a vertex $K\in \omega(K^{(1)})$. If $f$ is a homomorphism then we have $K^{(\rt)}f=\rt$. Let the  assumption holds  for first $q+1$ levels of $\omega(K^{(1)})$ and $Kf=x$ with $K\in U_{q}(\omega(K^{(1)}))$, $x\in \omega(1)$. Since $f$ preserves the genders and parent-child relations, we have  $\s(K)f=\s(x)$ and first statement of the lemma for the left subtrees is proved.
 \par Assume the converse: let $K\notin\omega(K^{(1)})$ and  $Kf=y$ for $y\in \omega(1)$. Then  $K$ is a male vertex. Let $P\in \omega(K^{(1)})$ be the lowest ancestor of $K$.   Thus, there exists female ancestor $F\in T(\widetilde{K})$ of $K$  and $K\prec F\prec P$. Therefore, we get that the female  vertex $Ff$ is an ancestor of $Kf\in\omega(1)$ which is impossible.

If $\rt=n$ then it remains nothing to prove. Otherwise the proof is dual.
 \end{proof}
\begin{lem}\label{lemma_korrektnost'} For each decreasing tree $(\overline{n},\prec)$, a convex partition $\widetilde{K}$ and an interval tree $T(\widetilde{K})$ associated with $(\overline{n},\prec)$, there exists a unique homomorphism $f\colon T(\widetilde{K})\to (\overline{n},\prec)$.

\end{lem}
\begin{proof}

We construct a  homomorphism $f\colon T(\widetilde{K})\to (\overline{n},\prec)$ by induction on the level of $K_q\in T(\widetilde{K})$. Since $K^{(\rt)}$ is always can be defined, the base clearly holds: $K^{(\rt)}f=\rt$.

\par Let $k$ be the depth of $T(\widetilde{K})$. Suppose  $f$ is well-defined  for all vertices of $T(\widetilde{K})$ whose level is less than $q+1$, $0\leq q+1<k$. Let $Kf=x$ and $K\in U_{q}$. For the concreteness assume that $x<\rt$.
 Denote by  $y$ the leading element of the interval $K$. Note that by the construction of the partition tree we have $\lev(y)\geq\lev(x)$. There are following cases to consider.

\textit{ Case 1.} Suppose $x\in K$. Since the level of the leading element of the interval is the least one of the elements in $K$ and $\lev(y)\geq\lev(x)$, we get that  $x$ is the leading element of $K$. Thus,  $K$ has the son (the daughter)   in $T(\widetilde{K})$ whenever there exists the son (the daughter) of $x$. Therefore if $\s(K_q)$, $\dau(K_q)$ exist then $\s(K_q)f=\s(x)$, $\dau(K_q)f=\dau(x)$ as required.


\textit{ Case 2.} $x\notin K$. Thus we have $x\ne\rt$ and $\lev(y)> \lev(x)$.

 a) If $K\in \omega(K^{(1)})$ then by Lemma~\ref{lemma_path_to_path} $x\in \omega(1)$. Since $K\in \omega(K^{(1)})$, then by Lemma~\ref{lemma leadings of omega} we have $y\in \omega(1)$. The tree is decreasing, therefore right inner tree of $y$ subordinates the right inner tree of $ x$. Thus  $K$ has a child (children) whenever $y$ has a child (children), which in turn holds whenever  $x$ has a child (children). Thus $\s(K)f=\s(x)$, $\dau(K)f=\dau(x)$ are defined and unique.

b) Suppose $K\notin \omega(K^{(1)})$, whence $y\notin \omega(1)$. By Lemma \ref{lemma_path_to_path} it holds also $x\notin \omega(1)$.

\par 1.   Let $y\prec x$. Since $(\overline{n},\prec)$ is decreasing, the inner trees of $y$ subordinate the respective inner trees  of $x$. Hence, $\s(K_q)f$, $\dau(K_q)f$ are well-defined and unique.

\par 2.    Assume $y\nprec x$. Denote by $a, b\in \omega(1)$, $a
 \ne b$ the lowest ancestors of $y$ and $x$ respectively. We have $a\prec b$  and $\Gamma_r(K^{(a)})\hookrightarrow \Gamma_r(a)\hookrightarrow\Gamma_r(b)$. Thus $\s(K)f=\s(x)$, $\dau(K)f=\dau(x)$ are well-defined.

\par 3. Let  $y\nprec x$ and  $x,y$ have the same lowest ancestor, say $a\in \omega(1)$.  Denote by $t$ the lowest common ancestor of $x,y$ with $t\ne a$. Note that there always  exists at least one  ancestor $\s(a)$ of $x,y$ which satisfies these conditions.

\par Again, without loss of generality for concreteness we assume $y<t<x$. Furthermore, by the inductive assumption the pre-images of $\p(x),\ldots, t\in\omega(x)$ have been already defined. Let $K'f=t$. Since $f$ is a homomorphism of order-preserving trees, $f$ is order-preserving itself. Thus, $K'f=t<Kf=x$ implies $t'<y$ for the leading element $t'$ of $K'$. The level of $t'$ cannot be less than $t$, therefore $t'\prec t$.
Then we have $\Gamma_r(t')\hookrightarrow\Gamma_r(t)$, whence $\s(K)f=\s(x)$, $\dau(K)f=\dau(x)$ are well-defined.
 \end{proof}

\begin{definition}\label{def_varphi} Let $(\overline{n},\prec)$ be a decreasing  tree. Fix  a convex partition $\widetilde{K}$ of~$\overline{n}$ into $m$ blocks, $1\leq m\leq n$. Let $f: T(\widetilde{K})\to (\overline{n},\prec)$ be a homomorphism. Denote by \textbf{$\varphi^{\widetilde{K}}$} the  transformation of $\overline{n}$ with the partition $\widetilde{K}$ such that  $K_i\varphi^{\widetilde{K}}=\{x\varphi^{\widetilde{K}}\mid x\in K_i\}=K_if$ for all  $K_i\in \widetilde{K}$.
\end{definition}

\begin{example} Let $T_3(5)$,  $\widetilde{K}$ be as in Example~\ref{ex_partition_tree}. Since  $\{3,4\}f=3$,  $\{1,2\}f=\s(3)=1$, $\{5\}f=\dau(3)=4$, we have
$\varphi^{\widetilde{K}}=\left(
\begin{array}{ccc}
\{12\}&\{34\}&\{5\}\\
1    &3     &4
\end{array}
\right)$.


\begin{table}
\caption{Semigroup $\Phi_{\prec}$ for $T_3(5)$}
\label{SemPhi}
\begin{tabular}{c|c||c|c}
$\widetilde{K}$& $\varphi^{\widetilde{K}}_{\prec}$ &$\widetilde{K}$& $\varphi^{\widetilde{K}}_{\prec}$\\
\hline
\hline
$\{1\}\{2\}\{3\}\{4\}\{5\}$&
$\left(
\begin{array}{ccccc}
\{1\}&\{2\}&\{3\}&\{4\}&\{5\}\\
1&2&3&4&5
\end{array}
\right)$&  $\{12345\}$&
$\left(
\begin{array}{c}
\{12345\}\\
3
\end{array}\right)$ \\
\hline\hline
$\{1\}\{2\}\{3\}\{45\}$&
$\left(\begin{array}{cccc}
\{1\}&\{2\}&\{3\}&\{45\}\\
1     &2   &3    &4
\end{array}
\right)$&$\{12\}\{345\}$&
$\left(
\begin{array}{ccc}
\{12\}&\{34\}&\{5\}\\
1&    3   & 4
\end{array}
\right)$\\
\hline
$\{1\}\{2\}\{34\}\{5\}$&
$\left(\begin{array}{cccc}
\{1\}&\{2\}&\{34\}&\{5\}\\
1    &2    &3&    5
\end{array}
\right)$&$\{123\}\{45\}$&
$\left(
\begin{array}{cc}
\{123\}&\{45\}\\
3&4
\end{array}
\right)$ \\
\hline
$\{1\}\{23\}\{4\}\{5\}$&
$\left(\begin{array}{cccc}
\{1\}&\{23\}&\{4\}&\{5\}\\
1    &3    &4&5
\end{array}
\right)$&$\{1234\}\{5\}$&
$\left(
\begin{array}{cc}
\{1234\}&\{5\}\\
3&4
\end{array}
\right)$ \\
\hline
$\{12\}\{3\}\{4\}\{5\}$&
$\left(\begin{array}{cccc}
\{12\}&\{3\}&\{4\}&\{5\}\\
1&3&4&5
\end{array}
\right)$&
$\{1\}\{2345\}$&
$\left(
\begin{array}{cc}
\{1\}&\{2345\}\\
1&3
\end{array}
\right)$
 \\
\hline
\hline
$\{1\}\{23\}\{45\}$&
$\left(
\begin{array}{ccc}
\{1\}&\{23\}&\{45\}\\
1&3&4
\end{array}
\right)$&
$\{1\}\{234\}\{5\}$&
$\left(
\begin{array}{ccccc}
\{1\}&\{234\}&\{5\}\\
1&3&4
\end{array}
\right)$ \\
\hline
$\{12\}\{3\}\{45\}$&
$\left(
\begin{array}{ccc}
\{12\}&\{3\}&\{45\}\\
1&3&4
\end{array}
\right)$&
$\{12\}\{34\}\{5\}$&
$\left(
\begin{array}{ccc}
\{12\}&\{34\}&\{5\}\\
1&3&4
\end{array}
\right)$ \\
\hline
$\{123\}\{4\}\{5\}$&
$\left(
\begin{array}{ccc}
\{123\}&\{4\}&\{5\}\\
3&4&5
\end{array}
\right)$ &$\{1\}\{2\}\{345\}$&
$\left(
\begin{array}{ccc}
\{1\}&\{2\}&\{345\}\\
1&2&3
\end{array}
\right)$\\
\hline
\end{tabular}
\end{table}
\end{example}

We denote by $\Phi_{\prec}$ the set of  transformations $\varphi^{\widetilde{K}}$ for all possible convex partitions  $\widetilde{K}$  of $\overline{n}$. The  set $\Phi_{\prec}$ for $T_3(5)$ is presented in  Table \ref{SemPhi}.

\begin{lem}\label{lemma_Phi_is_R}
 The set $\Phi_{\prec}$ constitutes an $\R$-cross-section of semigroup $\On_n$.
\end{lem}
\begin{proof} Let $(\overline{n},\prec)$ be a fixed order-preserving tree on $\overline{n}$, $T(\widetilde{K})$ be the order-preserving tree of a fixed partition $\widetilde{K}$ of $\overline{n}$. Since no other decreasing tree appears in the proof, an element of  $\Phi_\prec$ is  completely determined by the partition $\widetilde{K}$ of $\overline{n}$. Thus we will write $\varphi^{\widetilde{K}}$ for brevity. Let $K\in \widetilde{K}$, and $K\varphi^{\widetilde{K}}=v$. Then according to Definition~\ref{def_varphi}
\begin{gather*}
\s(K)=K'\Longleftrightarrow K'\varphi^{\widetilde{K}}=\s(v),\\
\dau(K)=K''\Longleftrightarrow K''\varphi^{\widetilde{K}}=\dau(v),
\end{gather*}
where $K',K''\in \widetilde{K}$. Since both trees are order-preserving, it holds $\varphi^{\widetilde{K}}\in\On_n$.

 By the construction $\Phi_{\prec}$ contains exactly one representative from each $\R$-class of $\On_n$. It remains to show that $\Phi_{\prec}$ constitutes a semigroup.
 \par Let $A$, $B$ be the convex partitions of $\overline{n}$ into $m$ and $t$ intervals respectively,  $\alpha=\varphi^{\widetilde{A}}$ and $\beta=\varphi^{\widetilde{B}}\in~\Phi_{\prec}$. Denote by $\widetilde{AB}$  the partition $\overline{n}/\ker{(\alpha\beta)}$ consisting of $k$ blocks:
  $$\widetilde{AB}=\{(\im{\alpha}\cap B_i)\alpha^{-1}\mid \im{\alpha}\cap B_i\ne\emptyset,\ B_i\in \widetilde{B}, 1\leq i\leq t\}.$$

We claim that  $\alpha\beta=\varphi^{\widetilde{AB}}$. First note that $\overline{n}/\ker(\alpha\beta)=\widetilde{AB}$.  We will show that $P(\alpha\beta)=P\varphi^{\widetilde{AB}}$ for all $P\in\widetilde{AB}$. We proceed by induction on level $q$ of a vertex $P\in T(\widetilde{AB})$.

 Let $q=0$. Clearly, the root of $T(\widetilde{AB})$ contains the root of $(\overline{n},\prec)$. By the definition  $\rt\alpha=\rt\beta=\rt$. If $P^{(\rt)}$ denotes the root of $T(\widetilde{AB})$  then $P^{(\rt)}\alpha\beta=(\rt)\alpha\beta=\rt=P^{(\rt)}\varphi^{\widetilde{AB}}$  as required.

 Assume that $P(\alpha\beta)=P\varphi^{\widetilde{AB}}$ for all $P$ in $q$-th level of $T(\widetilde{AB})$ for a natural $q$ with $q<s<k$.
 If $P$ has no children there is nothing to prove. Without loss of generality assume that $S\in T(\widetilde{AB})$ is the son of $P$. Our aim to show that $S\alpha\beta$ is the son of $P\alpha\beta.$
 \par Denote by $m$ the leading element of $S$.  Clearly,  $S=A_{k_1}\cup A_{k_2}\cup\ldots\cup A_{k_j}$ for $A_{k_i}\in \widetilde{A}$ with $A_{k_i}\alpha\in B^{(m\alpha)}$. Thus $S\alpha\beta=(m\alpha)\beta$. Note that the parent $\p(m\alpha)\in\im(\alpha)$ since $\alpha$ is induced by a homomorphism of trees. Furthermore, assume that $\p(m\alpha)\in B^{(m\alpha)}$. Again, by the definition of $\Phi_{\prec}$ and since $m\alpha\prec \p(m\alpha)$, we get then a contradiction with $m$ is the leading element of $S$. Therefore, $P$ contains the interval $A\in \widetilde{A}$ with $A\alpha=\p(m\alpha)$. So, we get $P\alpha\beta=(\p(m\alpha))\beta$.

 On the other hand, it is clear that $B^{(m\alpha)}$ is the son of $B^{(\p(m\alpha))}$ in $T(\widetilde{B})$. Thus $\beta\in \Phi_{\prec}$ implies $(m\alpha)\beta$ is the son of $(\p(m\alpha))\beta$, whence we get immediately that $S\alpha\beta$ is the son of $P\alpha\beta$. Therefore $\alpha\beta=\varphi^{\widetilde{AB}}$ as required.
 \end{proof}
\par The  converse statement also holds: each $\R$-cross-section of $\On_n$ coincides with  $\Phi_{\prec}$ for a decreasing  binary tree.

\subsection{Proof of the  converse statement}

First we need several auxiliary results.
Consider an $\R$-cross-section $R$ of $\On_n$. We begin with  simple properties of $R$.

\begin{lem}\label{lemma_nep_t} The following statements hold:
 \begin{enumerate}
 \item[(i)] Each $\R$-cross-section  of $\On_n$ has a fixed point.
 \item[(ii)] An $\R$-cross-section of $\On_n$ has at most two fixed points: $1$ and $n$.
 \end{enumerate}
\end{lem}
\begin{proof}
 \par (i)
Let $\const_r\in R$ be a constant transformation. Since for all $\alpha\in R$ it holds $\const_r\alpha=\const_{r\alpha}\in R$, and thus $\const_{r\alpha}=\const_r\alpha$, we get $r\alpha=r$.
\par (ii) Note that the points $1$ and $n$ are the only ones that always belong to different classes of each 2-element convex partition of  $\overline{n}$. Therefore, if $R$ has two fixed points then the only possible ones are $1$ and $n$.
     \end{proof}

\begin{definition}
Given an $\R$-cross-section $R$ of $\On_n$ and its fixed point $r$, construct a strictly increasing chain
\begin{equation}\label{chain}
W_0=\{r\}\subset W_1\subset W_2\subset \ldots \subset W_t=\overline{n}
\end{equation}
 of subsets of $\overline{n}$ by the rule: $$ W_i=W_{i-1}\cup\{x\in\im{(\alpha)}\mid \alpha\in R,\ \ |\im{(\alpha)}|=i+1,\ x\notin W_{i-1}\}, \  1\leq i< n.$$
\end{definition}
In virtue of Lemma~\ref{lemma_nep_t} the chain always has at least two components.
 \par  By  $W_i'$ denote the set $W_0$ if $i=0$ and $W_{i}\setminus W_{i-1}$ if  $1\leq i\leq t$.

Recall that a transversal of a partition is called  a set that contains exactly one representative from each block of the partition.

\begin{definition}\label{def_set_theta} Let $v$ be a vertex of an order-preserving tree $(\overline{n},\prec)$ with  $|\omega(v)|=k$. Denote by $\Theta_k^v\subset \Phi_{\prec}$ the set of all transformations $\theta^v$ such that $\omega(v)$  is a transversal of $\overline{n}/\ker(\theta^v)$.
\end{definition}
We will write also just $\Theta^v$ if the length of the path does not matter for us. Note that by the definition of $\Phi_\prec$, we get immediately $\im(\theta^v)=\omega(v)$ for all $\theta^v\in \Theta^v$. Thus,  $\Theta^v\subset \Phi_{\prec}$ is a  set of full idempotents whose image is $\omega(v)$.  Furthermore, it is easy to see that  $\Theta^v$ is a semigroup of left zeroes for each $v\in(\overline{n},\prec)$.

\begin{lem}\label{lemma_stroenie W} Let $R$ be an $\R$-cross-section  of $\On_n$. Then the sets $W_0$, $W_1'$, $\ldots$, $W_t'$ form the levels of an order-preserving tree such that for every $v\in W_{t}$ it holds $\Theta^v\subset R$ .
\end{lem}
 \par \begin{proof}  We  induct on the number of levels $t$.

 The base of induction holds. Indeed, if $\const_r\in R$ then $W_0=\{r\}$. Thus the root of required binary tree $\rt=r$, $1\leq r\leq n$. In this case $\theta^r=c_r\in R$.

  Suppose the lemma holds for a natural $s$, $1\leq s<t$: there defined a unique order-preserving tree with the vertex set $W_{s}$. The  tree's levels as sets are exactly  $W_0', W'_1,\ldots,$ $ W'_{s}$. Also we have $\Theta^v\subset R$ for all $v\in W_s$.  Let $v\in W_s'$ and  $a,b\in \overline{n}$ be the canonical bounds of $v$ as a vertex of the  tree, $a< b$.

  Without loss of generality we consider the left interval $[a,  v-1]$. Observe that if $v=1$ or $a\in \omega(v)$ and $v=a+1$ then there is nothing to prove. Otherwise we will construct an interval $K\subseteq [a,  v-1]$ with $K\cap W_s=\emptyset$; and  $\alpha\in R$ such that $K\subset\overline{n}/\ker(\alpha)$ and $K\alpha\in W_{s+1}'\cap K$. Then we show that $|W'_{s+1}\cap K|=1$.

   \textit{Case 1.} Let $a\notin \omega(v)$, therefore $a=1$. Let $K:=[1,v-1]$. By the definition of order-preserving binary tree $K\cap W_s=\emptyset$. Let $\alpha\in R$ be such that $K\in \overline{n}/\ker(\alpha)$ and the set $\{1\}\cup \omega(v)$ is a transversal of $\overline{n}/\ker(\alpha)$. Since $|\omega(v)|=s+1$ we get $|\im(\alpha)|=s+2$. Hence $K\alpha\in W'_{s+1}$. By the assumption $\Theta^v\subset R$. Let $\theta^v\in \Theta^v$.  We have  $\ker(\theta^v\alpha)=\ker(\theta^v)$, therefore $\theta^v\alpha=\theta^v$, whence $\alpha|_{\omega(v)}=\id_{\omega(v)}$. Since $\alpha$ is order-preserving, we get $K\alpha<v\alpha=v$. Thus, $K\alpha\in W'_{s+1}\cap K$.


   \textit{Case 2.} Let $a\in \omega(v)$. Then denote $K:=[a+1,v-1]$. Since $a$ is a canonical bound of $v$, we get $K\cap W_s=\emptyset$. Consider  $\alpha\in R$ such that  $\{a+1\}\cup \omega(v)$ is a transversal of $\overline{n}/\ker(\alpha)$, $K\in \overline{n}/\ker(\alpha)$.  Again, we get $\ker(\theta^v\alpha)=\ker(\theta^v),$ whence $$a=a\alpha<K\alpha<v\alpha=v.$$ Thus $(a+1)\alpha\in W'_{s+1}\cap K$.

\par  Now we proceed for both cases. We will show that $|W'_{s+1}\cap K|=1$.
  Denote $K\alpha=x$ for $x\in K$ and suppose there exists $y\in K\cap W_{s+1}'$ such that  $y\ne x$. That is, there exists  $\gamma\in R$ with  $y\in\im(\gamma)$ and $|\im(\gamma)|=s+2$. Consider the composition $\gamma\alpha$. Clearly, $x\in\im(\gamma\alpha)$. By the definition of a chain $W$ the condition $x\in W'_{s+1}$ implies $|\im(\gamma\alpha)|=s+2$. Consequently, $\ker{(\gamma\alpha)}=\ker{(\gamma)}$, whence $\gamma\alpha=\gamma$, which contradicts $x\ne y$. Therefore, $x\in W_{s+1}'$ is unique. Thus, $x$ is the son of $v$ and it is well-defined.
  \par  In particular, the above proof imply that for all $\beta\in R$ such that $\{x\}\cup \omega(v)$ is a transversal of $\overline{n}/\ker(\beta)$, it holds $\omega(v)\beta=\omega(v)$ and  $x\beta=x$. Therefore, by Definition~\ref{def_set_theta}  $\beta\in\Theta^{x}$. Hence, $\Theta^{x}\subset R$.
\end{proof}


\begin{corollary}\label{sled_x_p(x)_in_im}
 Let $R$ be an $\R$-cross-section  of $\On_n$; $W_0$, $W_1'$, $\ldots$, $W_t'$ the levels of the order-preserving tree $(\overline{n}, \prec)$ associated with $R$. Then the following statements hold.
\begin{itemize}
\item [(a)] It holds $|W_{i}'|\leq 2^{i}$ for all $i$, $0\leq i\leq t$.
 For all  $\alpha\in R$ with $\im{(\alpha)}\supseteq W_i$ it holds  $|\im{(\alpha)}|\leq\sum_{k=0}^{i}2^k$.
\item[(b)] Let $\alpha\in R$, $v\in\im{(\alpha})$.
Then $\omega(v)\subseteq \im{(\alpha}).$
\end{itemize}
\end{corollary}
\begin{proof} (a) Since $W'_i=U_i$, the proof follows directly from the definition of a binary tree.

   \par (b) Suppose $v\in U_{k-1}\cap\im{(\alpha})$, $1\leq k\leq t+1$.
     The condition $v\in U_{k-1}$ implies $|\im(\alpha)|\geq k$. Let $\theta^v\in\Theta_{k}^v$ be a fixed transformation. Then we have   $v\in\im{(\alpha\theta^v)}$, hence $|\im(\alpha\theta^v)|\geq k$ for all $\theta^v\in\Theta_{k}^v$. Since $|\im(\theta^v)|= k$, then $\im{(\alpha)}$ contains a transversal of a partition of $\theta^v$. In virtue of arbitrariness $\theta^v\in \Theta_{k}^v$, we have
    $$\im{(\alpha)}\supseteq\{p_0, p_1, \ldots, p_{k-1}\}.$$
   \end{proof}

    We denote by $(\overline{n},\prec_R)$ the order-preserving binary tree associated with $R$.
\begin{lem}\label{lemma_nevozr_der}
 The order-preserving tree $(\overline{n},\prec_R)$ is decreasing.
\end{lem}
\begin{proof}
Case when $n=1,2$ is trivial. Assume $n>2$. Let $x,y\in\overline{n}$ be such that $x\prec y$. Let $x\in W_k$, $y\in W_m$, $k<m$. To prove that the tree is order-preserving we need to verify conditions 1)-3) of Definition~\ref{def_ord_pres}. Consider  the following cases:

1.  Let $x,y\in\omega(1)$.  We denote by $T_r(x)$ and $T_r(y)$ the right-side subtrees of $(\overline{n},\prec_R)$ rooted at $\s(x)$ and $\s(y)$ respectively (the respective tree is empty if the vertex  has no son). To prove that condition 1) of Definition~\ref{def_ord_pres} holds  we will  show that there exists a homomorphism between $T_r(x)$ and $T_r(y)$. Apparently, the existence of  the homomorphism  implies that the right inner tree of $x$ subordinates the right inner tree of $y$. We will construct the transformation $\alpha\in R$ such that if $T_r(x)\ne\emptyset$ then  $T_r(x)\alpha\subseteq T_r(y)$. We will show that $\alpha|_{T_r(x)}$ is a homomorphism.

Let $\alpha\in R$ be such that $[x,\p(x)-1]\cup\omega(\p(y))$ are the leading elements of the partition $\overline{n}/\ker(\alpha)$. Since $\p(x), y$ are higher than $x$, we get  $\p(x)\alpha=y\alpha=\p(y)\alpha$.

We claim that $x\alpha=y$, $\p(x)\alpha=\p(y)$. Indeed, let $\theta^x\in \Theta_{k+1}^x$, $\theta^{\p(y)}\in \Theta_{m}^{\p(y)}$. Note that $|\im{(\theta^x\alpha)}|=m+1$. Thus, $x\alpha\in W_{m}$. Further, we have $\ker(\theta^{\p(y)}\alpha)=\ker(\theta^{\p(y)})$. Consequently, $\p(x)\alpha=y\alpha=\p(y)\alpha=\p(y)$.
By the definition of an order-preserving tree and since $\alpha\in \On_n$, $y$  is the only element in $W_m$ which is less than $\p(x)\alpha=\p(y)$. Therefore, $x\alpha=y$.

 If $x$ has no daughter then $T_r(x)$ is empty, whence $T_r(x)$ subordinates $T_r(y)$ by the definition. Suppose $|T_r(x)|\geq1$. We will show that $T_r(x)\hookrightarrow T_r(y)$. We induct on the level of a vertex of $T_r(x)$. Consider the root $\dau(x)$ of the tree. We claim that $\dau(x)\alpha\in W_{m+1}$. Indeed, the condition $|\im(\theta^{\dau(x)}\alpha)|=|\im(\theta^{x}\alpha)|+1=m+2$, implies the required. Moreover, since $\alpha\in \On_n$, we get $$x\alpha=y<\dau(x)\alpha<\p(y)=\p(x)\alpha.$$
   Whence the only possibility is $\dau(x)\alpha=\dau(y)$ and the induction base holds.
   \par Suppose   $\alpha$ induces  a homomorphism $T_r(x)\hookrightarrow T_r(y)$ for the first  $l$ levels of $T_r(x)$. Note that as was shown above $f$ preserves the canonical bounds of the vertices of the subtree. Let $v\in T_r(x)\cap W_{k+l}$ be a vertex, $v_1, v_2$ be its canonical bounds, $v_1<v_2$. By the assumption we have $v\alpha\in T_r(y)\cap W_{m+l}$.  Without loss of generality suppose $v$ has a son. Then we have $|\im(\theta^{\s(v)}\alpha)|=|\im(\theta^{v}\alpha)|+1$. Therefore,  $\s(v)\alpha\in W_{m+l+1}$ and $v_1\alpha<\s(v)\alpha<v\alpha$. Hence, $\s(v)\alpha$  is the son of $v\alpha$ by the definition of an order-preserving tree. Consequently, $T_r(x)\hookrightarrow T_r(y)$.

\par   If $y>\rt$  and $x,y\in\omega(n)$ then in dual way $T_l(x)\hookrightarrow T_l(y)$, where  $T_l(x)$ and $T_l(y)$ the left-side subtrees of $(\overline{n},\prec_R)$ rooted at $\s(x)$ and $\s(y)$ respectively. Hence condition 2) of Definition~\ref{def_ord_pres} also holds.

2. Let $x,y\notin\omega(1)\cup \omega(n)$. We denote by $T(x)$ and $T(y)$ the  subtrees of $(\overline{n},\prec_R)$ rooted at $x$ and $y$ respectively. Note that to prove condition 3) of Definition~\ref{def_ord_pres} it is enough to show that there exists a homomorphism between $T(x)$ and $T(y)$. As an the previous item we will construct the transformation $\beta\in R$ such that $T(x)\beta\subseteq T(y)$,  $\beta|_{T_a}$ is a homomorphism. Without loss of generality suppose $x<y<\rt$. The condition $x\prec y$ implies that there exists a minimal ancestor $c\in\omega(1)$ for both $x,y$. Furthermore, there are no $t\in \omega(y)$ with $x<t<y$. Thus, we have $c<x<y<\p(c)$. Moreover, let $a,b$ be the canonical bounds of $x$, $a<b$; $q,h\in\omega(y)$ be the canonical bounds of $y$. Then it holds $q\leq a<b\leq h$.
Let $\beta\in R$ be a transformation such that $\omega(\p(y))\cup T(x)$ are the leading elements of partition $\overline{n}/\ker(\beta)$. Then
 it holds $q\beta=a\beta$ and  $b\beta =y\beta=h\beta$, since $a,b,y,h$ are higher than $x$.

 Let $\theta^{\p(y)}\in \Theta_{m}^{\p(y)}$. Then $\theta^{\p(y)}\beta=\theta^{\p(y)}$. Thus, the set $T(x)\beta$ is bounded by the canonical bounds of $y$, i.\,e. $T(x)\beta\subseteq T(y)$.  Then $|\im(\theta^x\beta)|=|\omega(\p(y))+1|$, whence $x\beta\in W_{m}$. Thus, the only possibility is $x\beta=y$.

 If $T(x)$ is 1-element then the required is proved. Suppose $\beta$ induces a homomorphism of the first $l$ levels of $T(x)$ to  $T(y)$.  Let $v\in T(x)\cap W_{m+l}$; $v_1, v_2$ be the canonical bounds of the vertex, $v_1<v_2$.  Without loss of generality suppose $v$ has the son. Just as in the previous item one can show that $\s(v)\beta$ defines the son of $v\beta\in T(y)$. Dually,  $\dau(v)\beta$ defines the daughter of $v\beta\in T(y)$.
 Therefore, $T(x)$ subordinates $T(y)$ as required.

If $y<x<\rt$ then it holds $q\beta=y\beta=a\beta$ and  $b\beta =h\beta$. The further proof is exactly the same. If  $\rt>y$ we also get the same cases.\end{proof}

Now we are ready to prove
\begin{lem}\label{lemma_obratnoe}
For every $\R$-cross-section $R$ of $\On_n$ coincides with $\Phi_\prec$ for a fixed decreasing binary tree  $(\overline{n},\prec)$.
\end{lem}
\par \begin{proof} Consider a $(t+1)$-level decreasing binary tree $(\overline{n},\prec_R~)$ that arises with an $\R$-cross-section $R$ of $\On_n$. We claim that a transformation $\alpha\in R$  with a partition $\widetilde{K}$ has the form $\varphi^{\widetilde{K}}\in\Phi_{\prec_R}$. The proof is by induction on the level $U_i$ of the partition tree  $T(\widetilde{K})$.
 \par  Suppose $i=0$. Then  $K^{(\rt)}\alpha=\rt$ by Lemma~\ref{lemma_nep_t}. So the base of induction holds.  Assume that there exists $k$, $0\leq k\leq t$, such that $\alpha|_{U_0\cup\ldots\cup U_{k-1}}=\varphi^{\widetilde{K}}|_{U_0\cup\ldots\cup U_{k-1}}$.
\par   Let $X\in U_{k-1}$.  By the assumption  $X\alpha=X\varphi^{\widetilde{K}}$. Let $X\alpha=x$, $a,b$ be the canonical bounds of $x$, $a<b$. By the construction of $\varphi^{\widetilde{K}}$ we have $x\in W_{m}'$, $m\leq k-1$. If $X$ has no children then there is nothing to prove.
 Without loss of generality suppose $X$ has the son $S\in T(\widetilde{K})$. Let $S\alpha=y$. Since $\alpha\in \On_n$ we get $a, x$ are the canonical bounds of $y$. In other words, $y$ is a vertex of the subtree of $(\overline{n},\prec_R)$ rooted at $\s(x)$. Therefore, $\omega(x)\subseteq \omega(y)$. Note that for all $v\in S$ we have $y\in\im(\theta^v\alpha)$. On one hand, we have  $|\im(\theta^v\alpha)|=|\omega(x)|+1$. On the other hand, by Corollary~\ref{sled_x_p(x)_in_im} we get $\omega(y)\subseteq \im(\theta^v\alpha)$. Therefore $y=\s(x)=S\varphi^{\widetilde{K}}$ as required.
\end{proof}


Thus, Lemmas \ref{lemma_Phi_is_R} and \ref{lemma_obratnoe} imply immediately the following fact.
\begin{thm}\label{th_description}
Let $(\overline{n},\prec)$ be a decreasing  binary tree. Then semigroup $\Phi_{\prec}$  is an  $\R$-cross-section of $\On_n$. Conversely, each $\R$-cross-section of $\On_n$ is given by the semigroup $\Phi_{\prec}$ for a decreasing  binary tree.
\end{thm}

\section{On $\Lg$-cross-sections of $\On_{n}$ and $\R$-cross-sections of $\On_{n+1}$}\label{sec_L_cr_sec}
We have pointed  already that there is a connection between $\Lg$-cross-sections of $\T_{n}$ and   $\R$-cross-sections of $\On_{n+1}$. In this section we discuss the connection in more detail. A key notion in the description of $\Lg$-cross-sections of $\T_{n}$ plays a notion of a respectful tree (the definition  will be stated below). It turns out, that the structure of an $\R$-cross-section of $\On_n$ has an alternative interpretation  in terms of respectful trees.
\par \subsection{Respectful trees}

Recall the notion of a respectful tree as it was proposed in \cite{Bondar_Volkov_CRA}.
 If $u$ and $v$ are two vertices of the same tree $T$, we say that $u$ subordinates $v$ if the subtree rooted at $u$ subordinates the subtree rooted at~$v$.

\begin{definition}
A full binary tree is said to be \emph{respectful} if it satisfies two
conditions:
\par (S1) if a male vertex has a nephew, the nephew subordinates his uncle;
\par (S2) if a female vertex has a niece, the niece subordinates her aunt.
\end{definition}
For an illustration, the tree shown in Fig.~\ref{fig:tree} satisfies (S1) but fails to satisfy (S2): the daughter of the root has a niece but this niece does not subordinates her aunt. On the other hand, the tree shown in Fig.~\ref{fig:respectful tree} is respectful. (In order to ease the insperation of this claim, we have shown the uncle--nephew and the aunt--niece relations in this tree with dotted and dashed arrows respectively.)

\begin{figure}[h]
\begin{center}
\unitlength 0.50mm
\begin{picture}(80,90)(0,5)
\gasset{AHnb=0}
\gasset{Nw=10,Nh=10,Nmr=5}
\node(1)(10,30){}
\node(2)(20,10){}
\node(3)(40,10){}
\node(4)(30,30){}
  \drawedge(2,4){}
  \drawedge(3,4){}
\node(5)(20,50){}
  \drawedge(1,5){}
  \drawedge(4,5){}
\node(6)(50,30){}
\node(7)(70,30){}
\node(8)(60,50){}
  \drawedge(6,8){}
  \drawedge(7,8){}
\node(9)(40,70){}
  \drawedge(5,9){}
  \drawedge(8,9){}
\node(10)(70,70){}
\node(11)(55,90){}
  \drawedge(9,11){}
  \drawedge(10,11){}
\end{picture}
\caption{An example of a non-respectful tree}\label{fig:tree}
\end{center}
\end{figure}
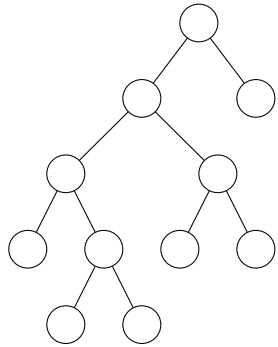

\begin{figure}[h]
\begin{center}
\unitlength 0.50mm
\begin{picture}(100,90)(0,5)
\gasset{AHnb=0}
\gasset{Nw=10, Nh=10, Nframe=y, Nfill=n}
\node(1)(10,30){}
\node(2)(20,10){}
\node(3)(40,10){}
\node(4)(30,30){}
  \drawedge(2,4){}
  \drawedge(3,4){}
\node(5)(20,50){}
  \drawedge(1,5){}
  \drawedge(4,5){}
\node(6)(50,30){}
\node(7)(70,30){}
\node(8)(60,50){}
  \drawedge(6,8){}
  \drawedge(7,8){}
\node(9)(40,70){}
  \drawedge(5,9){}
  \drawedge(8,9){}
\node(10)(90,70){}
\node(11)(55,90){}
  \drawedge(9,11){}
  \drawedge(10,11){}
\node(12)(80,50){}
\node(13)(100,50){}
  \drawedge(12,10){}
  \drawedge(13,10){}
\gasset{AHnb=1,dash={0.2 1.2}0}
  \drawedge[sxo=7,syo=-3,exo=-6,eyo=5](9,12){}
  \drawedge[sxo=7,syo=-3,exo=-7,eyo=4](5,6){}
  \drawedge[sxo=4,syo=-6,exo=-2,eyo=4](1,2){}
\gasset{AHnb=1,dash={1.5}{1.5}}
  \drawedge[curvedepth=-2,exo=3.5,eyo=5.5](10,8){}
  \drawedge[curvedepth=-2,exo=3.5,eyo=5.5](8,4){}
\end{picture}
\caption{An example of a respectful tree}\label{fig:respectful tree}
\end{center}
\end{figure}
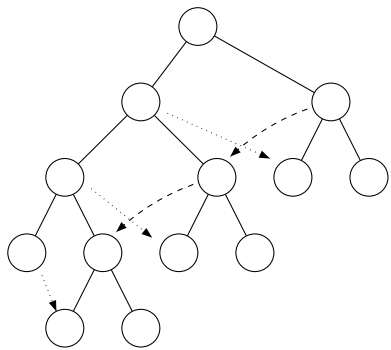
  To preserve the notations of previous papers concerning respectful trees we will denote the respectful tree  by $\Gamma$.  There are also equivalent definition of respectful tree (see, \cite{Bondar2016}). Using current terminology the one can be stated as follows:
 \begin{pro}[\cite{Bondar2016}, Proposition 1]\label{Pr_altern_def} A full binary tree $\Gamma$ is respectful if and only if every non-root vertex subordinates his (her) parent.
 \end{pro}

 \subsection{The faithful interval marking of a respectful tree}
We have already used the intervals as labels of binary tree vertices when were spoken about the inner trees and the partition trees. In our construction for $\Lg$-cross-sections of $\T_n$,  we also use certain markings  of respectful trees by intervals of the set $\overline{n}$  considered as a chain under a fixed linear order:
 $$u_1\prec u_2\prec \ldots\prec u_n.$$

 \par Let $\Gamma$ be a fixed respectful tree.
If $i,j\in\overline{n}$ and $i\leq j$, the \emph{interval} $[u_i,u_j]$ is the set $\{k\in \overline{n}\mid u_i\preceq k\preceq u_j\}$. We write $[u_i]$ instead of $[u_i,u_i]$. Now, a \emph{faithful interval marking} of a tree $\Gamma$ is a map $\mu_{\prec}$ from the vertex set of $\Gamma$ into the set of all intervals in $\overline{n}$ such that for each vertex $v$,
\begin{itemize}
\item  the number of leaves in the subtree $\Gamma(v)$ rooted at $v$ is equal to the number  of elements in the interval $v\mu_{\prec}$;
\item if $v\mu_{\prec}=[u_p,u_q]$ for some $u_p,u_q\in\overline{n}$ with $u_p\prec u_q$, and $s$ and $d$ are respectively the son and the daughter of $v$, then $s\mu_{\prec}=[u_p,u_t]$ and $d\mu_{\prec}=[u_{t+1},u_q]$ for some $t$ such that $p\le t<q$.
\end{itemize}
It easy to see that given a fixed linear order $\prec$ every respectful tree $\Gamma$ admits a unique faithful interval marking $\mu_{\prec}$.

Fig.~\ref{fig:marking} demonstrates a faithful interval marking   $\mu_{<}$ for usual order $<$ of the tree from Fig.~\ref{fig:respectful tree}.
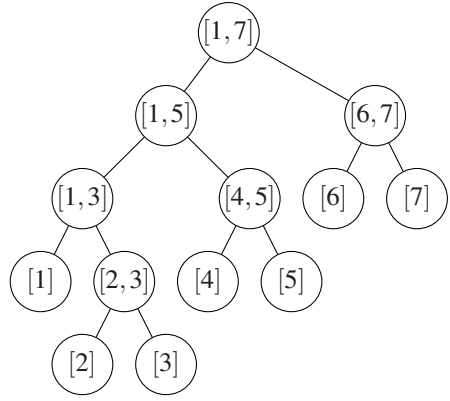
\begin{figure}[hb]
\begin{center}
\unitlength 0.55mm
\begin{picture}(100,81)(0,5)
\gasset{AHnb=0}
\node(1)(10,30){$[1]$}
\node(2)(20,10){$[2]$}
\node(3)(40,10){$[3]$}
\node(4)(30,30){$[2,3]$}
  \drawedge(2,4){}
  \drawedge(3,4){}
\node(5)(20,50){$[1,3]$}
  \drawedge(1,5){}
  \drawedge(4,5){}
\node(6)(50,30){$[4]$}
\node(7)(70,30){$[5]$}
\node(8)(60,50){$[4,5]$}
  \drawedge(6,8){}
  \drawedge(7,8){}
\node(9)(40,70){$[1,5]$}
  \drawedge(5,9){}
  \drawedge(8,9){}
\node(10)(90,70){$[6,7]$}
\node(11)(55,90){$[1,7]$}
  \drawedge(9,11){}
  \drawedge(10,11){}
\node(12)(80,50){$[6]$}
\node(13)(100,50){$[7]$}
  \drawedge(12,10){}
  \drawedge(13,10){}
\end{picture}
\caption{A faithful interval marking of the tree from Fig.~\ref{fig:respectful tree}}\label{fig:marking}
\end{center}
\end{figure}
\par For the respectful tree $\Gamma$ with  its faithful interval marking $\mu_{\prec}$ we will not make a difference between  a vertex of the tree and the interval associated with it.
 \subsection{$\Lg$-cross-sections of $\On_n$}
   First we recall the  description of $\Lg$-cross-sections of $\T_n$ for convenience.
   Let $\Gamma$ be a respectful tree with  its faithful interval marking.
   For each $M\subset\overline{n}$ we construct inductively the transformation with the image $M$,  using partial transformations, whose domains go through vertices of  $\Gamma$ bottom up.

 \par For functions $f$ and $g$ with disjoint domains, we denote by $f\cup g$ the union of $f$ and $g$ (viewed as sets of pairs). In other words, if $h = f\cup g$, then $\dom{(h)} = \dom{(f)}\cup \dom{(g)}$ and for all $x\in\dom{(h)}$, $xh = xf$ if $x \in \dom{(f)}$, and $xh = xg$ if $x\in \dom{(g)}$.
  \par Let  $M\subset\overline{n}$, by  $\langle M\rangle\in \Gamma$ we denote the intersection of all vertices of  $\Gamma$ which contains $M$. Given $\langle M\rangle=B$ and $A\in\Gamma$, we define the mapping $\alpha^{A}_M$  inductively as follows:
\begin{itemize}
\item[(a)] if $A=\emptyset$ then $\alpha^{A}_M =\emptyset$ (empty mapping);
\item[(b)] if $M=\{m'\}$ and $A\ne\emptyset$, then $\dom{(\alpha^{A}_M)}=A$ and $x'\alpha^{A}_M=m'$ for every $x'\in A$;
\item[(c)] if $|M|>1$ and $A\ne\emptyset$, then $\alpha^{A}_M=\alpha^{\s(A)}_{M\cap \s(B)}\cup\alpha^{\dau(A)}_{M\cap \dau(B)}$.
 \end{itemize}
   We denote by $\alpha_M$ the full transformation $\alpha^{\overline{n}}_M$. The description of $\Lg(\T_n)$ gives the following
\begin{thm}[\cite{Bondar2014}, Theorem 1]\label{th_descr_LT} For every respectful tree $\Gamma$ with $n$ leaves and its faithful interval marking $\mu_{\prec}$, the set $L^{\Gamma}=\{\alpha_M\mid M\subset \overline{n}\}$ forms an $\Lg$-cross-section of $\T_n$. Conversely, every $\Lg$-cross-section of $\T_n$ is given by $L^{\Gamma}$ for a respectful tree $\Gamma$ with a faithful interval marking.
\end{thm}
It remains to recall the notion of similar respectful trees (see~ \cite{Bondar2016}). We have already define a homomorphism of the binary trees. It is natural to define \emph{an isomorphism } between two trees as a bijective homomorphism between the ones.

By \emph{an anti-isomorphism} between two trees $T_1$ and $T_2$ we mean the bijective mapping that sends root of $T_1$ to the root of $T_2$,  preserves  the parent-child relations and  inverses the genders of the vertices.
\begin{definition}
 Recall that respectful trees $\Gamma_1$, $\Gamma_2$ are called \textit{similar} (write $\Gamma_1\sim\Gamma_2$)   if there exists either isomorphism or anti-isomorphism between the trees.
 \end{definition}
In other words, $\Gamma_1\sim\Gamma_2$  if  $\Gamma_1$ is a mirror reflection of $\Gamma_2$.
\begin{example}\label{ex_sim_trees}
Consider an example of similar respectful trees.
\begin{figure}[hb]
\begin{center}
\unitlength 0.55mm
\begin{picture}(100,75)
\put(10,75){$\Gamma_1$}
\gasset{AHnb=0}
\node(1)(0,10){$[1]$}
\node(2)(20,10){$[2]$}
\node(3)(30,30){$[3]$}
\node(4)(10,30){$[1,2]$}
  \drawedge(1,4){}
  \drawedge(2,4){}
\node(5)(20,50){$[1,3]$}
  \drawedge(3,5){}
  \drawedge(4,5){}
\node(6)(40,50){$[4]$}
\node(7)(30,70){$[1,4]$}
  \drawedge(5,7){}
  \drawedge(6,7){}
\put(80,75){$\Gamma_2$}
\gasset{AHnb=0}
\node(1)(130,10){$[4]$}
\node(2)(110,10){$[3]$}
\node(3)(100,30){$[2]$}
\node(4)(120,30){$[3,4]$}
  \drawedge(1,4){}
  \drawedge(2,4){}
\node(5)(110,50){$[2,4]$}
  \drawedge(3,5){}
  \drawedge(4,5){}
\node(6)(90,50){$[1]$}
\node(7)(100,70){$[1,4]$}
  \drawedge(5,7){}
  \drawedge(6,7){}
\end{picture}
\caption{A  pair of similar respectful trees}\label{fig:sim_trees}
\end{center}
\end{figure}
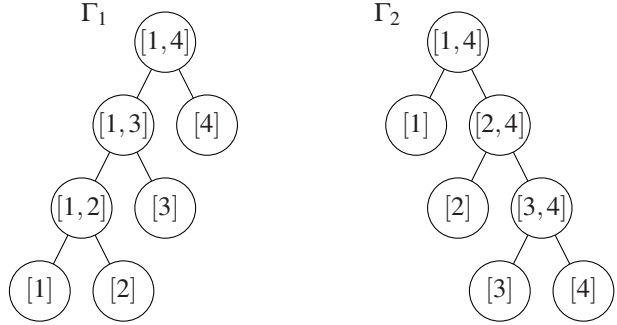

\end{example}
\par We need to recall also the following lemma.
\begin{lem}[\cite{Bondar2016}, Lemma 2]\label{lem_L1=L2Bondar2016}
 Let $\Gamma_1$, $\Gamma_2$ be the respectful trees with $n$ leaves and the faithful interval markings $\mu_{\prec_1}$, $\mu_{\prec_2}$ respectively. Then the $\Lg$-cross-sections associated with $\Gamma_1$, $\Gamma_2$ are coincide if and only if one of the following conditions holds:
 \begin{itemize}
\item[(i)] $\Gamma_1=\Gamma_2$;
\item[(ii)] there exists an anti-isomorphism from $\Gamma_1$ to $\Gamma_2$ and $\prec_2 = \prec_1^{-1}$.
\end{itemize}
\end{lem}
Now we are ready to describe $\Lg(\On_n)$.
\begin{thm}\label{th_descr_L_On} Let  $\Gamma$ be respectful tree with $n$ leaves and a faithful interval marking $\mu_{<}$, where $<$ is the  usual order on $\overline{n}$. The set $L^{\Gamma}=\{\alpha_M\mid M\subset \overline{n}\}$ forms an $\Lg$-cross-section of $\On_n$. Conversely, every $\Lg$-cross-section of $\On_n$ is given by $L^{\Gamma}$ for a respectful tree $\Gamma$ with a faithful interval marking $\mu_{<}$.
\end{thm}
\begin{proof} Indeed, the natural order on $\overline{n}$ provides that the vertices of $\Gamma$ with $\mu_{<}$ are the convex intervals. Consequently, each $\alpha_M\in L^{\Gamma}$ has a convex partition. Moreover, if $A\in\Gamma$  then by the definition of $\mu_{<}$ it holds $x<y$ for all $x\in\s(A), y\in\dau(A)$. Thus by the construction $\alpha_M$ is order-preserving for every $M\subseteq\overline{n}$.  Therefore $L^{\Gamma}$ is an $\Lg$-cross-section of $\On_n$ by Lemma~\ref{lemma_cr_sec_of_reg_subsem}.
 \par Conversely, it is clear that for all $M\subseteq\overline{n}$ there exists $\alpha\in\On_n$ with $\im(\alpha)=M$. Thus, every $\Lg$-cross-section of $\On_n$ is an $\Lg$-cross-section of $\T_n$. Therefore, it has the form $L^{\Gamma}$ for a respectful tree $\Gamma$ with a faithful interval marking $\mu_{\prec}$. We claim that $\overline{n}$ is naturally ordered, i.e. $\mu_\prec=\mu_<$. On the one hand, for every $A\in \Gamma$ there exits $\alpha_M\in \On_n$ with $A\in\overline{n}/\ker(\alpha_M)$. Therefore, $A$ is convex. On the other hand, since $\alpha_A\in \On_n$ for all $A\in\Gamma$ with $|A|\ne 1$, by  the definition of $\alpha_A$ we have that $x<y$ for all $x\in\s(A), y\in\dau(A)$. Thus, either  $\prec$ is equal to $<$, or $\prec$ is equal to $<^{-1}$. By Lemma~\ref{lem_L1=L2Bondar2016}, (ii) without loss of generality we may assume $n$ is  naturally ordered.
 \end{proof}

 \subsection{Elementary trees}
We define now a special type of decreasing trees, called elementary, which, we will see, generate in a sense all decreasing trees.
 \begin{definition} Say a subtree $T^x$ of the decreasing tree  to be \textbf{\emph{elementary}}, if $U_0\cup U_1=\{x,\p(x)\}$ and for all $v\in T^x$, $|T^x|>2$ it holds either $x<v<\p(x)$ or $\p(x)<v<x$.
\end{definition}
\par Looking at the diagram of an order-preserbing tree (decreasing, in particular) it is easy to see that the tree can be considered as a union of elementary trees. However, at the moment we are interested in a special case when the decreasing tree is elementary itself. That is, it holds either $(\overline{n},\prec)=T^1$ or $(\overline{n},\prec)=T^n$. In this case, clearly, the elementary trees correspond to $\R$-cross-sections of $\On_n$ with two fixed points. Again, if an $\R$-cross-section has two fixed points then by Definition~\ref{def_varphi} we get $U_0\cup U_1=\{1,n\}$. Thus the cross-section corresponds to one of the elementary trees  $T^1$, $T^n$.

 On the other hand, by Proposition~\ref{pro_dual_L} each dual to an $\Lg$-cross-section of $\On_{n}$ gives us the pair of $\R$-cross-sections of $\On_{n+1}$  with two fixed points. In turn, each $\Lg$-cross-section is determined by a respectful tree. We will see now that a respectful tree determines the pair of elementary trees on $[n+1]$.
\begin{lem}\label{lemma_resp_imply_elem} Given a respectful tree $\Gamma(n)$ with $n$-leaves, there exists a unique (up to the choice of the root) diagram of $(n+1)$-element elementary trees $T^1$, $T^{n+1}$  such that $\Gamma$ is the common inner tree of the trees.
\end{lem}
\begin{proof}
  We induct on the number $n$ of leaves of $\Gamma(n)$. Let $n=1$. Obviously $T^1$, $T^{2}$ are the elementary trees with with the common inner tree $\Gamma(1)$.

 \par Let $k$ be  a natural number. Suppose the claim holds  for  all natural $l<k$. Consider a respectful tree $\Gamma(k)$. Without loss of generality denote the root of $\Gamma(k)$ by $\rt':=[1',k']$.

 Consider the binary trees $T^1$, $T^{k+1}$  such that  $U_0\cup U_1=\{1, k+1\}$ and $\Gamma(k)$ is the inner  tree for both of trees. For concreteness consider  $T^{k+1}$. Let $\dau(1):=b$. Let there exist $a\in T^{k+1}$ with $y\prec x$. Without loss of generality assume that $a<b$ (see Fig.~\ref{fig:resp_to_order_pres}). Then the left inner tree of $a$ subordinates the left inner tree of $b$ by Proposition~\ref{Pr_altern_def}. The right-hand  inner tree of $a$ is the niece of $[b',k']$. Hence, the right-hand  inner tree of $a$ subordinates the right-hand inner tree of $b$. Analogously the inner trees of $b$'s daughter (if the daughter exists) subordinates respective inner trees of $b$.  The order-preserving trees whose inner trees are $\Gamma_l(x)$ and $\Gamma_r(x)$  are  decreasing by the induction assumption. Therefore, we conclude that  $T^{k+1}$ is decreasing and hence, elementary.

 \end{proof}
\begin{figure}[h]
\begin{center}
\begin{picture}(40,40)
\gasset{AHnb=0,linewidth=0.4}
  \drawline[linewidth=0.1](0,10)(40,10)
  \drawline[linewidth=0.1](0,15)(40,15)
  \drawline[linewidth=0.1](0,20)(40,20)
 \drawline[linewidth=0.1](0,25)(40,25)
  \drawline[linewidth=0.1](0,30)(40,30)
  \drawline[linewidth=0.1](0,35)(40,35)
  \drawline[linewidth=0.1](0,40)(40,40)
  \put(-1,43){$1$}
  \put(10,43){$a$}
  \put(20,43){$b$}
  \put(39,43){$k+1$}
  \drawline(0,10)(0,40)
  \drawline(5,10)(5,40)
  \drawline(10,10)(10,40)
  \drawline(15,10)(15,40)
  \drawline(20,10)(20,40)
  \drawline(25,10)(25,40)
  \drawline(30,10)(30,40)
  \drawline(35,10)(35,40)
  \drawline(40,10)(40,40)
  \gasset{AHnb=0,linewidth=0.8, Nw=1.5,Nh=1.5,Nframe=y, Nfill=y, NLdist=3}
  \node(r_21)(0,40){}
  \node(r_22)(20,30){}
  \node(r_23)(40,35){}
  \node(r_31)(15,20){}
  \node(r_32)(10,25){}
   \node(r_32)(5,20){}
  \node(r_33)(35,25){}
  \node(r_33)(30,20){}
  \node(r_33)(25,15){}
    \drawline(0,10)(0,40)
    \drawline(5,10)(5,20)
    \drawline(10,10)(10,25)
    \drawline(15,10)(15,20)
    \drawline(20,10)(20,30)
    \drawline(25,10)(25,15)
    \drawline(30,10)(30,20)
    \drawline(35,10)(35,25)
    \drawline(40,10)(40,35)
\gasset{AHnb=1,dash={0.2 1.2}0}
 \drawline(0,30)(20,30)
 \drawline(0,25)(10,25)
 \drawline(20,25)(35,25)
\gasset{AHnb=1,dash={1.5}{1.5}}
 \drawline(20,30)(40,30)
 \drawline(35,25)(40,25)
 \drawline(10,25)(20,25)
  \end{picture}
\caption{A connection between respectful and decreasing trees}\label{fig:resp_to_order_pres}
\end{center}
\end{figure}
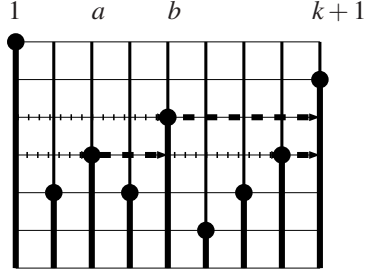

The converse statement also holds: given an elementary tree determines the respectful tree.
\begin{lem}\label{lemma_inner tree is resp} Let $(\overline{n}, \prec)$ be a decreasing binary tree, $x, \p(x)\in\overline{n}$. The inner tree of elementary tree $T^x$  is respectful.
\end{lem}
\par \begin{proof}
Without loss of generality assume $\p(x)<x$. Consider the inner tree $\Gamma$ of $T^x$. We induct on the number leaves of $\Gamma$. If $|\Gamma|\leq2$ then there is nothing to prove. Let $\Gamma$ have $k$ leaves, $k>2$. Denote by $b$ the son of $x$.  Without loss of generality assume that $b$ has the son $a$ (again, one may use for the illustration Fig.~\ref{fig:resp_to_order_pres} having in mind that $U_0\cup U_1=\{\p(x),x\}$. By the  definition of a decreasing tree,  we have that the inner trees of $a$ subordinate the respective inner trees of $b$, which implies directly that the  niece $[a',(b-1)']\in \Gamma$ subordinates her aunt $[b', (x-1)']\in \Gamma$. In the same way if there exists the nephew of $[\p(x)',(b-1)']$ then the nephew subordinates his uncle.
\par Apparently, $\Gamma_l(b)$ and $\Gamma_r(b)$ have less than $k$ leaves. To conclude that $\Gamma_l(b)$, $\Gamma_r(b)$ are respectful it remains to point out the elementary trees whose inner trees are $\Gamma_l(b)$ and $\Gamma_r(b)$. The tree $T^b\subset T^x$ is elementary. It's inner tree is exactly $\Gamma_r(b)$. Thus, by the induction assumption $\Gamma_r(b)$ is respectful. Consider the tree $T$ whose first levels $U_0\cup U_1=\{\p(x), b\}$ and the inner tree coincides with $\Gamma_l(b)$. By the condition the diagram determines an order-preserving tree. Thus, $T$ is elementary. Therefore, by the induction assumption the inner tree is respectful which completes the proof.

\end{proof}

We need now to  generalize the notion of canonical bounds for the vertices of a partition tree with respect to an elementary tree. Let $T^x$ be an elementary tree, $\widetilde{K}$ be a convex $(t+1)$-element partition of $[\p(x), x]$, $t>1$. Consider the partition tree $T(\widetilde{K})$. Let $k$ be the depth of the tree. Since $T^x$ is elementary, the vertices of first two  levels of $T(\widetilde{K})$ always contains the end-points of the interval $[\p(x), x]$. Therefore, we define the canonical bounds for the vertices  $X\in U_i$, $2\leq i\leq k$. We induct on the level $i$ of a vertex $X\in\widetilde{K}$. Note that in the case we have always $|U_2|=1$. Without loss of generality assume $\p(x)<x$. Let $i=2$, then define $K^{(\p(x))}<X<K^{(x)}$. Let $X\in U_q$, $2\leq q<k$. Assume $A$, $B\in T(\widetilde{K})$ be the canonical bounds of $X$,  $A<X<B$. If $X$ has the son $\s(X)$ then $A<\s(X)<X$. If $X$ has the daughter $\dau(X)$ then $X<\dau(X)<B$.

Now we will see that due to the connection between respectful and elementary trees, having a respectful tree one can get immediately the diagram of the  dual $({L^{\Gamma}})^*$.

\begin{lem}\label{lemma_image_alpha_dual}
Let $\Gamma$ be the inner tree of an elementary tree $([n+1],\prec)$. Then
  $$({L^{\Gamma}})^*=\Phi_{\prec}\setminus\{\const_r\}.$$
\end{lem}

\begin{figure}[h]
\center{\includegraphics[width=0.5\linewidth]{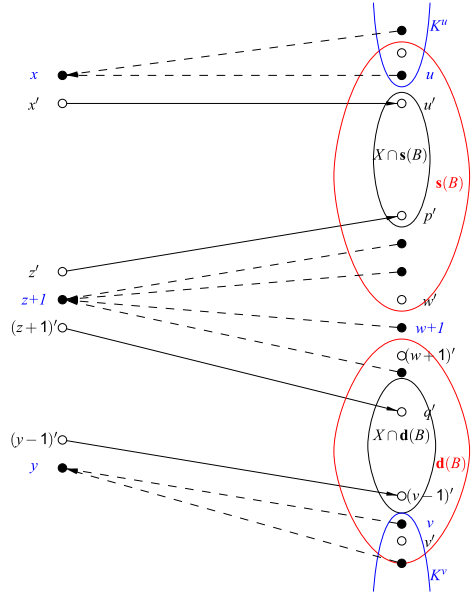}}
\caption{Proof of Lemma \ref{lemma_image_alpha_dual}, the step of induction.}
\label{fig_proof}
\end{figure}

\begin{proof} Consider the $\Lg$-cross-section $L^{\Gamma}$, associated with $\Gamma$ and the $\R$-cross-section $\Phi_\prec$, associated with $([n+1],\prec)$.
 Let $\alpha_M\in L^{\Gamma}$ with $M=\{r_1',r_2',\ldots, r_t'\}$. Then $M$ induces the (t+1)-element convex partition $\widetilde{K_M}$ of $[n+1]$:
$$\widetilde{K_M}=[1,r_1]\cup[r_1+1,r_2]\cup\ldots\cup[r_t+1,n+1].$$
 \par We will  show now that $\alpha_M^*=\varphi^{\widetilde{K_M}}$.
 We proceed the proof  by induction on the level $i$ of the tree $T(\widetilde{K_M})$, $0\leq i\leq k$. The induction base holds since by the definition of a dual transformation, we always have $K^{(1)}\alpha_M^*=1$ and  $K^{(n+1)}\alpha_M^*=n+1$.

Let the statement holds for the for first  $i$ levels of $T(\widetilde{K_M})$.  Consider a vertex $K^{(u)}\in U_{i}$ of $T(\widetilde{K_M})$. Let $K^{(w)}, K^{(v)}$ be the canonical bounds of the vertex. Without loss of generality consider $K^{(u)}$ with its right canonical  bound $K^{(v)}$. We also may assume that $u$ is the right end-point of $K^{(u)}$ and $v$ is the left end-point of $K^{(v)}$ or redenote the intervals if otherwise holds. Then  $u',(v-1)'\in \im(\alpha)$ (see Fig.~\ref{fig_proof}).

 By the assumption we have $K^{(u)}\alpha^*_M=K^{(u)}\varphi^{\widetilde{K_M}}$ and $K^{(v)}\alpha^*_M=K^{(v)}\varphi^{\widetilde{K_M}}$. Let $K^{(u)}\alpha^*_M=x$, $K^{(v)}\alpha^*_M=y$. Since $\varphi^{\widetilde{K_M}}$ is induced  by the homomorphism of the $i$-levels subtrees between $T(\widetilde{K_M})$ and $([n+1], \prec)$, we get that $y$ is the right canonical bound of $x$. Therefore by the construction of the inner tree $\Gamma$ it holds $A:=[x',(y-1)']\in\Gamma$.

 The condition $|A|=1$, i.e. $x=y-1$, implies that $u'=(v-1)'$,  therefore $K^{(u)}$ has no daughter and there is nothing  to prove.  If $|A|>1$ and $u'=(v-1)'$ there is also nothing  to prove. Let $X=\im(\alpha)\cap[u',(v-1)']$ and $|X|>2$. Let $\langle A\alpha_M\rangle=B\in\Gamma$.

  Let $z+1\in ([n+1],\prec)$ be the vertex induced by $\s(A)$ and $\dau(A)$. That is,
    $$\s(A)=[x',z'],\ \dau(A)=[(z+1)', (y-1)'].$$
    Similarly, let $w+1\in ([n+1],\prec)$ be the vertex induced by $\s(B)$ and $\dau(B)$. Therefore, $K^{(w+1)}$ is the daughter of $K^{(u)}$ and $K^{(w+1)}\varphi^{\widetilde{K_M}}=z+1$.

 Let $$ X\cap \s(B)=[u',p'],\ X\cap \dau(B)=[q',(v-1)'].$$
   By the definition $\alpha^{A}_X=\alpha^{\s(A)}_{X\cap \s(B)}\cup\alpha^{\dau(A)}_{X\cap \dau(B)}$, whence $z'\alpha^{A}_X=p'$, $(z+1)'\alpha^{A}_X=q'$.
 Consequently $p+1, q\in K^{(w+1)}$ and $K^{(w+1)}\alpha_M^*=z+1$ as required.
\end{proof}
\par Therefore, we get the following fact.

\begin{thm}\label{th_dual}  The $\R$-cross-section of $\On_{n+1}$ has two fixed points if and only if it is produced by an $\Lg$-cross-section of $\On_{n}$.
\end{thm}
\par \begin{proof} Necessity. Suppose $R$ has two fixed points and let  $([n+1],\prec)$ be the decreasing tree associated with $R$. As we observed the  tree $([n+1],\prec)$ is elementary. Thus, on one hand $\const_{\rt}\in\{\const_1,\const_{n+1}\}$. On the other hand by Lemma~\ref{lemma_inner tree is resp} the inner tree $\Gamma$ is respectful. Therefore, by previous lemma $\Phi_{\prec}=(L^{\Gamma})^*\cup\{\const_{\rt}\}$.

Proposition~\ref{pro_dual_L} implies immediately the sufficiency.
\end{proof}

\subsection{Decomposition of the decreasing trees}

\par Clearly, every order-preserving tree can be presented as a union of elementary trees in many ways. Consider the set $\omega(\overline{n},\rt)=\omega(1)\cup\omega(n)$ ordered with respect to $<$. We will denote it also by $\omega(R)$. The set is defined by the tree in a unique way and always has at least  two elements $1$ and $n$.
We can present  the original tree as a union of the following elementary  subtrees (see Fig.~\ref{fig:decomposition}):
\begin{equation}\label{decomposition}
T^{1}, T^{\p(1)}, \ldots, T^{\s(\rt)}, T^{\dau(\rt)},\ldots,  T^{\p(n)}, T^{n}.
\end{equation}
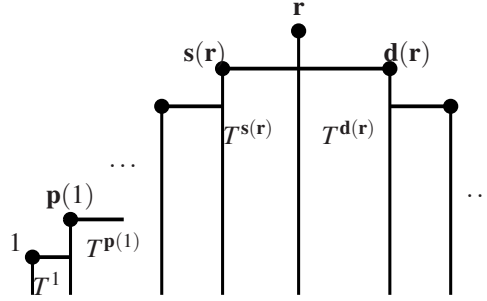
\begin{figure}[h]
\begin{center}
\begin{picture}(70,40)
\gasset{AHnb=0,linewidth=0.4}
  \drawline(35,5)(35,40)
   \drawline(25,5)(25,35)
  \drawline(17,5)(17,30)
    \put(10,22){$\ldots$}
  \drawline(5,15)(5,5)
  \drawline(0,10)(0,5)
  \drawline(47,5)(47,35)
  \drawline(55,5)(55,30)
  \put(57,18){$\ldots$}
  \drawline(62,5)(62,15)
  \drawline(70,10)(70,5)
\gasset{AHnb=0,linewidth=0.5, Nw=1.5,Nh=1.5,Nframe=y, Nfill=y, NLdist=3}
                 \node(r)(35,40){$\rt$}
  \node[NLangle= 140](s1)(25,35){$\s(\rt)$}
  \node[NLangle= 45](s2)(17,30){}
                \node(s3)(5,15){$\p(1)$}
  \node[NLangle= 140](s4)(0,10){$1$}
  \drawline(25,35)(35,35)
  \put(25,25){$T^{\s(\rt)}$}
  \drawline(17,30)(25,30)
  \drawline(12,15)(5,15)
  \put(7,10){$T^{\p(1)}$}
  \drawline(0,10)(5,10)
  \put(0,5){$T^{1}$}
\node[NLangle= 40](d1)(47,35){$\dau(\rt)$}
\node[NLangle= 60](d2)(55,30){}
             \node(d4)(70,10){$n$}
\drawline(47,35)(35,35)
\put(38,25){$T^{\dau(\rt)}$}
\drawline(47,30)(55,30)
\drawline(70,10)(62,10)
\put(64,5){$T^n$}
   \end{picture}
\caption{A minimal decomposition of an order-preserving tree}\label{fig:decomposition}
\end{center}
\end{figure}

On one hand, it is easy to see that  any other decomposition of $T(n)$ into elementary trees,  contains more components. On the other hand, for all $a\in\omega(R)$ there is no  elementary tree $T^x\subseteq (\overline{n},\preceq)$ such that  $T^a\subset T^x$, $x\in\overline{n}$. We say the decomposition (\ref{decomposition}) is \textit{minimal}.
\par Now the definition of decreasing tree and Lemma~\ref{lemma_inner tree is resp}  imply immediately  the following
 \begin{thm} Let $a_1=1<a_2<\ldots <a_t<\ldots<a_{k-1}<a_k=n$ be a  sequence of natural numbers with a marked point $a_t$, $1\leq t\leq k\leq n$. Let $n_i=a_{i+1}-a_i$, $\Gamma(n_i)$ denote a respectful tree on $n_i$-element set such that
  $$\Gamma(n_1)\hookrightarrow \Gamma(n_2)\hookrightarrow\ldots\hookrightarrow \Gamma(n_{t-1}),$$
 $$\Gamma(n_{k-1})\hookrightarrow \Gamma(n_{k-2}) \ldots \hookrightarrow \Gamma(n_t).$$ The sequence with a marked point $a_t$ and the trees define a unique decreasing tree.
 \end{thm}
Since every decreasing tree determines the $\R$-cross-section, we can regard the last theorem as the alternative description of $\R$-cross-sections of $\On_{n}$ in terms of respectful trees.

 In particular,  the sequence $\omega(\overline{n},\rt)=\{1,n\}$  with a  respectful tree,  define the dual  $\R$-cross-section up to the choice of the root.  The other special case is when $\omega(\overline{n},\rt)=\{1,2,\ldots,n\}$ with $\rt=1$ or $\rt=n$. We get in these cases the dense cross-sections.

\section{Classification of $\R$-cross-sections of $\On_n$ up to an isomorphism}
\subsection{Necessary conditions }\label{subsec_necess_cond}
 Let $R$ and  $\hat{R}$ be the $\R$-cross-sections of $\On_n$ associated with diagrams $(\overline{n},\prec_{1})$ and $(\overline{n},\prec_{2})$ respectively.

First we will look for necessary conditions two $\R$-cross-sections to be isomorphic.  We will see in the following lemma  that the  sets of idempotents $\Theta^x$ play an important role in the structure of $\R$-cross-sections.
 To avoid the ambiguous we  write $\hat{U_i}$ for the $i$-th level of $\hat{R}$.


\begin{lem}\label{lemma_Theta_x_phi=Theta_y} Let $R\stackrel{\xi}{\cong}\hat{R}$. Then for each $x\in U_{k}$ with $0\leq k< n$  there exists $y\in \hat{U}_{k}$ such that $$\Theta^{x}\xi=\Theta^{y}.$$
Moreover, the mapping $\sharp:(\overline{n},\prec_1)\to (\overline{n},\prec_2)\colon x\mapsto y$ if and only if  $\Theta^x\xi=\Theta^{y},$
      is well-defined, bijective and preserves the levels of the vertices.
\end{lem}
\begin{proof}  We proceed by induction on the level $i$ of a vertex $v\in R$.
Denote by $\rt_1$, $\rt_2$ the roots of $R$ and $\hat{R}$ respectively. Since $\const_{\rt_1}\xi=\const_{\rt_2}$, we have $\Theta^{\rt_1}\xi=\Theta^{\rt_2}$, $\rt_1^\sharp=\rt_2$, so the induction base holds.
\par Suppose the statement is true for all $U_i$, $\hat{U}_i$ with $1\leq i<k$.

Let $a\in U_{k-1}$. First note that $\omega(a^\sharp)=\omega(a)^\sharp$. Indeed, for all $b\in\omega(a)$, $\theta^{b}\in\Theta^{b}$ and only for them it holds $\theta^b\theta^a=\theta^b$. Therefore, $\im(\theta^{a^{\sharp}})=\{a^{\sharp},\p(a)^{\sharp},\ldots,\rt_1^{\sharp}\}=\omega(a)^\sharp$. On the other hand, $\im(\theta^{a^{\sharp}})=\omega(a^\sharp)$. Therefore,
\begin{equation*}\label{eq_w sharp}
\omega(a^\sharp)=\omega(a)^\sharp=\im(\theta^{a^{\sharp}}).
\end{equation*}
 \par Let $x\in U_k$ be the son (the daughter) of $a$. Consider the transformation $\theta^x\xi$. For all $c\in\omega(x)$ and only for them we have $\theta^{c}\theta^x=\theta^{c}$, whence by the induction assumption we get   $\omega(a^{\sharp})\subset\im{(\theta^x\xi)}$.  Let $y\in\im{(\theta^x\xi)}$, $y\notin\omega(a^{\sharp})$. Then  by Corollary~\ref{sled_x_p(x)_in_im} it holds $\omega(y)\subseteq \im{(\theta^x\xi)}$. Since for all $z\in\omega(y)$ it holds $\theta^z(\theta^x\xi)=\theta^z$, the only possibility is when

  $$\im{(\theta^x\xi)}=\omega(a^{\sharp})\cup\{y\} \mbox{ with }\p(y)\in \omega(a^{\sharp}).$$

In fact, we will show  that $y$ is either $\s(a^{\sharp})$, or $\dau(a^{\sharp})$. Suppose the converse. Then $\theta^x\xi\notin\Theta^t$ for all $t\in\overline{n}$. Since $\p(y)\in \omega(a^{\sharp})$, we get $y\in \hat{U}_m$, $m\leq k-1$. Therefore, by the induction assumption $y=v^{\sharp}$ for some $v\in U_m$.
  Then, on one hand, we have $\theta^x\theta^v\in \Theta^{\p(v)}$. Therefore, by the induction assumption we get $(\theta^x\theta^v)\xi\in \Theta^{\p(v)^\sharp}$.
  On the other hand, it holds $(\theta^x\theta^v)\xi=(\theta^x\xi)\theta^{v^{\sharp}}\in\Theta^{v^{\sharp}}$, which contradicts the assumption.
  Hence, we have  that $y$ is either $\s(a^{\sharp})$, or $\dau(a^{\sharp})$.

\par  Further, assume $\alpha,\beta\in \Theta^x$ are such that  $\alpha\xi=\theta^{\s(a^{\sharp})}$ and $\beta\xi=\theta^{\dau(a^{\sharp})}$.
Then  we have $\alpha\beta\in \Theta^x$ and $(\alpha\beta)\xi=(\alpha\xi)(\beta\xi)=\theta^{\s(a^{\sharp})}\theta^{\dau(a^{\sharp})}\in\Theta^{a^{\sharp}}$, we get a contradiction. Thus, we get either $\Theta^x\xi\subseteq \Theta^{\s(a^{\sharp})}$, or  $\Theta^x\xi\subseteq \Theta^{\dau(a^{\sharp})}$.

 \par Let $\alpha\in R$ be such that $\alpha\xi\in\Theta^{x^{\sharp}}$. Then $\alpha$ is idempotent and for all $\theta^x$ it holds $(\alpha\theta^x)\xi=\alpha\xi$. Therefore $\alpha\theta^x=\alpha$, whence $\im(\alpha)=\omega(x)$. Consequently, $\alpha\in \Theta^x$ as required and $x^\sharp$ is well-defined. Since $\xi$ is isomorphic we get $\sharp$ is bijective. \end{proof}

Consider the mapping $\sharp$ from the proof of Lemma~\ref{lemma_Theta_x_phi=Theta_y} in more detail.
\begin{corollary}\label{corol_sharp} Let $R\stackrel{\xi}{\cong}\hat{R}$. Then for all $a,b\in(\overline{n},\prec_{1})$, $\alpha\in R$ the following holds:
\begin{itemize}
\item [(1)]   $\omega(a^\sharp)=\omega(a)^\sharp$;
\item [(2)] the condition $a\prec b$ holds  if and only if $a^\sharp\prec b^\sharp$;
 \item [(3)] $\im(\alpha\xi)=\im(\alpha)^{\sharp}$.
 \end{itemize}
\end{corollary}
\begin{proof}
The first claim follows  immediately from the proof of Lemma~\ref{lemma_Theta_x_phi=Theta_y}. It remains to verify items (2)-(3).

(2) Indeed, it holds $$\theta^b\theta^a=\theta^b\Longleftrightarrow\Theta^b\Theta^a=\Theta^b\Longleftrightarrow \Theta^{b^\sharp}\Theta^{a^\sharp}=\Theta^{b^\sharp}\Longleftrightarrow a^\sharp\prec b^\sharp.$$

(3) Let $x\in \overline{n}$, $\beta\in R$ such that $\im(\beta)=\omega(x)$. Then for all $\theta^x\in\Theta^x$ it holds $(\beta\theta^x)\xi=\beta\xi=(\beta\xi)\theta^{x^\sharp}$. In virtue of arbitrariness of $\theta^x$, we get $\im(\beta\xi)=\omega(x^\sharp)$.

Let $\alpha\in R$ be such that $x\in \im{\alpha}$. By Corollary~\ref{sled_x_p(x)_in_im}, (b), we get $\omega(x)\subset \im{\alpha}$.
 Then it holds $\im(\alpha\theta^x)=\omega(x)$. Therefore for all $\theta^x\in\Theta^x$ $\im(\alpha\theta^x)\xi=\omega(x^\sharp)=\im(\alpha\xi)(\theta^{x^\sharp}),$
whence $\omega(x^\sharp)\subseteq\im(\alpha\xi)$.
 In virtue of arbitrariness $x$, we get $\im(\alpha\xi)=\im(\alpha)^{\sharp}$.
\end{proof}
\par We  also need the following lemma  in the sequel.

\begin{lem}\label{lemma_kardinality_Theta} Given a pair of isomorphic $\R$-cross-sections $R, \hat{R}$ of $\On_n$, $x\in \omega(R)$, the following statements hold:
 \begin{itemize}
\item [(1)] $|\Theta^{\rt}|= 1, |\Theta^{x}|=|\Theta^{\p(x)}|\cdot|\p(x)-x|,$
where $|\p(x)-x|$ is the absolute difference of $\p(x)$ and $x$.
\item [(2)] If $x^{\sharp}\in \omega(\hat{R})$ then $|\p(x)-x|=|\p(x)^{\sharp}-x^{\sharp}|$.
\end{itemize}\end{lem}
\begin{proof} (1) Clearly, $|\Theta^{\rt}|= 1$. Assume the statement holds for the first $k$ levels of the tree, $k<n.$ Let $x\in U_{k+1}\cap \omega(R)$.  Without loss of generality assume $x<\rt$. Denote by $P$ the set $\overline{n}\setminus[1,\p(x)-1]$.  Clearly, $\Theta^x|_P\cong \Theta^{\p(x)}$. Since for all $\theta^x\in\Theta^x_k$, we have $1\theta^x=\ldots=x\theta^x=x$,  the interval $K^{(x)}$ is totaly determined by its right-hand end-point. The condition $\p(x)\theta^x=\p(x)$ implies that there are $|\p(x)-x|$ possibilities to choose the interval $K^{(x)}$. Hence  $|\Theta^{x}|=|\Theta^{\p(x)}|\cdot(\p(x)-x)$ as required.

\par (2) Let  $x^{\sharp}\in \omega(\hat{R})$. By Corollary~\ref{corol_sharp} it holds $\p(x)^\sharp=\p(x^\sharp)$. Therefore $\p(x)^\sharp\in \omega(\hat{R})$. According to the previous item we have
$$|\Theta^{x}|=|\Theta^{\p(x)}|\cdot |\p(x)-x|=|\Theta^{\p(x^\sharp)}|\cdot|\p(x)^\sharp-x^\sharp|=|\Theta^{x^\sharp}|.$$
Since $|\Theta^{\p(x)}|=|\Theta^{\p(x)^\sharp}|$, we get  $|\p(x)-x|=|\p(x)^{\sharp}-x^{\sharp}|$.\end{proof}

Now we are ready to investigate the connection between the sets $\omega(R)$ and $\omega(\hat{R})$ of isomorphic cross-sections.
\begin{lem}\label{lemma seq_to_seq} Let $R\stackrel{\xi}{\cong}\hat{R}$. Then it holds  $\omega(R)^\sharp=\omega(\hat{R})$ and for all $a_i\in \omega(R)$, we have either $a_i=a_i^\sharp$  or $a_i+a_i^\sharp=n+1$.
\end{lem}

 \begin{proof}
 The proof is by induction on the level of a vertex.
 \par It is clear that $\rt_1^\sharp=\rt_2$. If $\rt_1\in\{1,n\}$ then  we get $|U_1|=1$. By Corollary~\ref{corol_sharp}, then $|\hat{U}_1|=1$, therefore  $\rt_2\in\{1,n\}$. Let $\rt_1\notin\{1,n\}$, $\dau:=\dau(\rt_1)$. Consider the set $D$  of $\alpha \in R$ with $\im(\alpha)=\{\rt_1,\dau\}$. By Corollary~\ref{corol_sharp}, (4), we have
 $$D\xi=\{\hat{\alpha} \in \hat{R}\mid \im(\hat{\alpha})=\{\rt_2, \dau^\sharp\}\}.$$

   We use the fact that $|D|=|D\xi|$. Clearly, $|D|=n-\rt_1$. For $D\xi$ we have two possibilities. If $1\leq \dau^\sharp<\rt_2$, we get $|D\xi|=\rt_2-1$.  If $\rt_2< \dau^\sharp\leq n$, we get $|D\xi|=n-\rt_2$. Hence,  we have either $\rt_1=\rt_2$  or $\rt_1+\rt_2=n+1$.


Suppose  the statement holds for $\omega(R)\cap U_m$ with $m\leq p$ for a natural $p$, $p<s$. Let  $a_i\in \omega(R)\cap U_p$. Without loss of generality assume  $\rt_1=\rt_2$ and $\rt_1<a_i\leq n$.
  By the assumption then we have  $a_i^\sharp\in \omega(\hat{R})$ and $a_i=a_i^\sharp$ (see  Fig.~\ref{fig:proof_of_lemma}). By Lemma~\ref{lemma_kardinality_Theta}, (2), it suffices to show that $a_{i+1}^\sharp\in\omega(\hat{R})$.
   By Corollary~\ref{corol_sharp}  we have that $a_{i+1}^\sharp$ is a child of $a_i^\sharp$. If $a_i-a_{i-1}=1$ then  $a_i^{\sharp}-a_{i-1}^{\sharp}=1$, so  $a_{i}^\sharp$ has a unique child. Thus, in this case the statement holds.

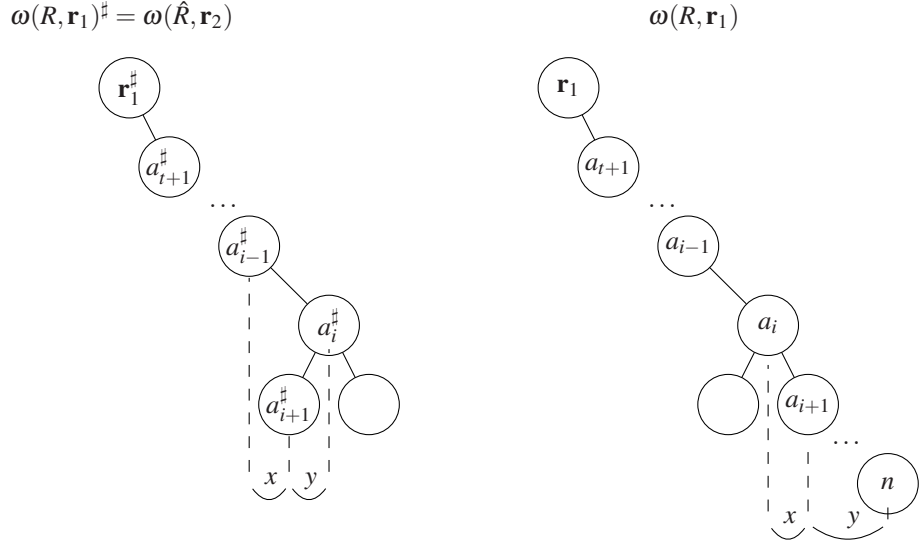
\begin{figure}[h]
\begin{center}
\unitlength=1.05mm
\begin{picture}(150,70)
\gasset{AHnb=0}
\put(5,68){$\omega(R,\rt_1)^\sharp=\omega(\hat{R},\rt_2)$}
\node[NLangle= 10](1)(20,60){$\rt_1^\sharp$}
\node[NLangle= 10](2)(25,50){$a_{t+1}^\sharp$}
\drawedge(1,2){}
\put(30,45){$\ldots $}
\node[NLangle= 10](3)(35,40){$a_{i-1}^\sharp$}
\node[NLangle= 10](4)(45,30){$a_{i}^\sharp$}
\drawedge(3,4){}
\node[NLangle= 0 ](5)(40,20){$a_{i+1}^\sharp$}
\drawedge(4,5){}
\node(6)(50,20){}
\drawedge(4,6){}
\put(85,68){$\omega(R,\rt_1)$}
\node[NLangle= 10](1')(75,60){$\rt_1$}
\node(2')(80,50){$a_{t+1}$}
\drawedge(1',2'){}
\put(85,45){$\ldots $}
\node[NLangle= 10](3')(90,40){$a_{i-1}$}
\node[NLangle= 20](4')(100,30){$a_{i}$}
\drawedge(3',4'){}
\node(5')(105,20){$a_{i+1}$}
\drawedge(4',5'){}
\node[NLangle=180](6')(95,20){}
\drawedge(4',6'){}
\put(108,15){$\ldots $}
\node[NLangle= 0](7')(115,10){$n$}
\gasset{AHnb=0,Nw=1.5,Nh=1.5,Nframe=n, Nfill=n, NLdist=3}
\node(x1)(35,10){}
\node(x2)(40,10){}
\node(x3)(45,10){}
\drawedge[curvedepth=-2](x1,x2){}\put(37,10){$x$}
\drawedge[curvedepth=-2](x2,x3){}\put(42,10){$y$}
\node(x1')(100,5){}
\node(x2')(105,5){}
\node(x3')(115,5){}
\drawedge[curvedepth=-2](x1',x2'){}\put(102,5){$x$}
\drawedge[curvedepth=-2](x2',x3'){}\put(110,5){$y$}
\gasset{AHnb=0,dash={1.5}{1.5}}
  \drawline(35,36)(35,10)
  \drawline(40,17)(40,10)
  \drawline(45,27)(45,10)
\drawline(100,25)(100,5)
  \drawline(105,15)(105,5)
  \drawline(115,7)(115,4)
\gasset{AHnb=1,Nw=1.5,Nh=1.5,Nframe=n, Nfill=n, NLdist=8}
\end{picture}
\caption{An illustration of the proof}\label{fig:proof_of_lemma}
\end{center}
\end{figure}
\par  Let $|a_i-a_{i-1}|=|a_i^{\sharp}-a_{i-1}^{\sharp}|>1$.  Suppose $a_{i+1}^\sharp\notin\omega(R)$. That  is, $a_{i+1}^\sharp$ is the son of $a_i^\sharp$. Hence, we have $a_{i-1}^\sharp<a_{i+1}^\sharp<b_i^\sharp$.
\par Suppose the left inner tree of $a_{i+1}$ is defined on $x$-element set (i.e., $a_{i+1}-a_i=x$) and the right tree of of $a_{i+1}$ is defined $y$-element set (i.e., $n-a_{i+1}=y$). By Corollary~\ref{corol_sharp}, (3), the inner trees associated with $a_{i+1}^\sharp$ are defined on  $x$- and $y$-element sets respectively:
 $$(a_{i+1}^\sharp-a_{i-1}^\sharp)
 (a_{i}^\sharp-a_{i+1}^\sharp)=xy,$$
 $$a_i^\sharp-a_{i-1}^\sharp=x+y.$$

Now, on the one hand, note that
 $$|\Theta^{a_{i+1}^\sharp}|=|\Theta^{a_{i-1}^\sharp}|xy=|\Theta^{a_{i-1}}|xy.$$
On the other hand, by the assumption it holds $a_i^\sharp-b_{i-1}^\sharp=a_i-a_{i-1}=x+y$, therefore we have
 $$|\Theta^{a_{i+1}}|= |\Theta^{a_{i-1}}|(a_i-a_{i-1})(a_{i+1}-a_i)=|\Theta^{a_{i-1}}|(x+y)x,$$
which contradicts the condition $|\Theta^{a_{i+1}}\xi|=|\Theta^{a_{i+1}}|$.
  Hence, $a_{i+1}^\sharp\in\omega(\hat{R})$. The second part of the statement of the lemma follows immediately from Lemma~\ref{lemma_kardinality_Theta},~(2).
 \end{proof}

 Thus, if $R\stackrel{\xi}{\cong}\hat{R}$ then ``the skeletons'' $\omega(R)$, $\omega(\hat{R})$ of the respective order-preserving trees are the same up to  a mirror reflection. Furthermore,  the elementary trees $T^{a}$ are also very ``close'' to $T^{a^\sharp}$,  $a\in\omega(R)$:
\begin{corollary}\label{corol_tree_to_tree} Let $R\stackrel{\xi}{\cong}\hat{R}$, then
 for all $a\in\omega(R)$ the mapping $\sharp\colon T^{a}\to T^{a^\sharp}$ is bijective and preserves the parent-child relations.
\end{corollary}
\begin{proof} Let $a\in\omega(R)$. By Lemma~\ref{lemma seq_to_seq} it holds $a^\sharp\in\omega(\hat{R})$. We have $(T^{a})^{\sharp}\cap \omega(\hat{R})=\{a^{\sharp},\p(a)^{\sharp}=\p(a^\sharp)\}$ with $|\p(a)-a|=|\p(a)^{\sharp}-a^{\sharp}|$. By Corollary~\ref{corol_sharp}, (2), we get now the required.
 \end{proof}
\par In fact, the stronger statement holds: the inner trees of $T^{a}$ and $T^{a^\sharp}$ are similar, for all $a\in\omega(R)$ (see Proposition~\ref{prop_clasif_short R}). To prove this we use the connection between the dual cross-sections with $\Lg$-cross-sections (see Sec.~5). Recall the following result from \cite{Bondar2016}:

    \begin{thm} [\cite{Bondar2016}, Theorem 3]\label{th_clas_L} Two $\mathscr{L}$-cross-sections of $\mathscr{T}_n$ are isomorphic if and only if the respectful trees associated with them are similar.
\end{thm}

Note that according to Corollary~\ref{corol_tree_to_tree}, without loss of generality we may assume the diagram of an $\R$-cross-section to be ``one-sided'', i.e., it holds $\omega(R)=\omega(n)$, for instance. Therefore,  we assume further $\rt_1=1$ and  consider below only  the right-hand  sided diagrams, unless otherwise stated.

\begin{pro}\label{prop_clasif_short R} If
 $R\stackrel{\xi}{\cong}\hat{R}$ then the inner trees of $T^{a}$, $T^{a^\sharp}$ are pairwise similar for all $a\in\omega(R)$.
\end{pro}
\begin{proof}
\par Let  $|\omega(R)|=k$, $a\in \omega(R)$ with $a\ne\rt$. Consider the elementary tree $T^{a}$. We denote the inner trees of $T^a$, $T^{a^\sharp}$ by $\Gamma^{a}$, $\Gamma^{a^\sharp}$ respectively. If  $|T^{a}|\leq 3$ then  Corollary~\ref{corol_tree_to_tree} implies directly that $\Gamma^{a}{\sim}\Gamma^{a^\sharp}$. Suppose $|T^{a}|=m> 3$.
 \par Let $\Phi_{\prec_1}$, $\Phi_{\prec_2}$ be $\R$-cross-sections associated with the elementary trees $T^a$, $T^{a^\sharp}$ respectively. We will construct now a subsemigroup $R(a)\subseteq R$ such that $R(a)\cong \Phi_{\prec_1}$.
\par Recall that the subsets of $[\p(a)',(a-1)']$ induces an interval partition of the set $[\p(a), a]$ (see Subsection~\ref{subsec_dual}). If $\p(a)=\rt$ then we denote by $R(a)$ the set  of transformations $\gamma\in R$ whose partition induced by a subset $V$ of $\Gamma^a$, for each $V\subset \Gamma^a$. If $\p(a)\ne\rt$ then let $b\in\omega(a)$ be the parent of $\p(a)$. In this case we denote by $R(a)$ the set  of transformations $\gamma\in R$ whose partition induced by $\{x'\mid x\in\omega(b)\}\cup V$, for each $V\subset \Gamma^a$. According to the description of $\R(O_n)$ we have $x\gamma=x$ for all $x\in\omega(b)$, $T^a\gamma\subseteq T^a$. Clearly, it holds
$$R(a)\cong \{\gamma|_{T^a}\colon\gamma\in R(a)\}\cong\Phi_{\prec_1}.$$

In the same manner we define the subset $R(a^\sharp)\subseteq R$ with  $R(a^\sharp)\cong \Phi_{\prec_2}$.
By Corollary~\ref{corol_sharp}, (4), it holds $\im(\gamma\xi)=\im(\gamma)^\sharp$. Since $\omega(v)^\sharp=\omega(v^\sharp)$ for all $v\in(\overline{n},\prec)$ and Corollary~\ref{corol_tree_to_tree}, we get $R(a)\xi=R(a^\sharp)$.

Trees $T^a$, $T^{a^\sharp}$ are elementary. Consequently, their inner trees are respectful. Let $L$, $\hat{L}$ be order-preserving $\Lg$-cross-sections associated with the inner trees. Using the fact that $R(a)\xi=R(a^\sharp)$ we will show that $L\cong \hat{L}$, whence by Theorem~\ref{th_clas_L} we get the required immediately.

For $\alpha\in R(a)$ denote by $\alpha_A\in L$, $A\in\Gamma^a$, the transformation such that $\alpha_A^*=\alpha|_{T^a}$. By $\hat{\alpha}_A\in \hat{L}$  we denote the transformation such that $\hat{\alpha}_A^*=\alpha\xi|_{T^{a^\sharp}}$.

We claim that a mapping
 $$\kappa\colon L\to\hat{L}\colon \alpha_A \mapsto\hat{\alpha}_A$$
is an isomorphism. Clearly, by the construction of $*$ and since $R(a)\xi=R(a^\sharp)$, we get that $\kappa$ is bijective. Let $\beta\in R(a)$ and $\alpha_B^*=\beta|_{T^a}$, for $B\in\Gamma^a$. Let $\alpha_{C}\in L$, $C\in\Gamma^a$, stand for $\alpha_A\alpha_B$. It remains to show that $\hat{\alpha}_{C}=\hat{\alpha}_A\hat{\alpha}_B$.
Note that
$$\alpha_{C}^*=(\alpha_A\alpha_B)^*=\alpha_B^*\alpha_A^*=(\beta|_{T^a})(\alpha|_{T^a}).$$
 Since $\alpha, \beta\in R(a)$ we get $(\beta|_{T^a})(\alpha|_{T^a})=(\beta\alpha)|_{T^a}$.
 Consequently, $\alpha_{C}^*=(\beta\alpha)|_{T^a}$. Then $$\hat{\alpha}_{C}^*=(\beta\alpha)\xi|_{T^{a^\sharp}}=(\beta\xi)(\alpha\xi)|_{T^{a^\sharp}}.$$
 Since $\alpha\xi, \beta\xi\in R(a^\sharp)$, we get $(\beta\xi)(\alpha\xi)|_{T^{a^\sharp}}=(\beta\xi|_{T^{a^\sharp}})(\alpha\xi|_{T^{a^\sharp}})=
 \hat{\alpha}_B^*\hat{\alpha}_A^*=(\hat{\alpha}_A\hat{\alpha}_B)^*$.
 Therefore, $\hat{\alpha}_{C}=\hat{\alpha}_A\hat{\alpha}_B$. Hence, $\kappa$ is isomorphism.

\end{proof}

Furthermore, let $T^{a\sharp}=T^b$. We claim that $\sharp\colon T^a\to T^b$ directly depends on the isomorphism/the anti-isomorphism between $\Gamma^a$ and $\Gamma^b$.
To prove this  we need to find a way how to ``synchronize'' the vertices of the elementary trees  according to the isomorphism of their inner trees.
\begin{definition}\label{def_pi}  We denote by $\psi$ an isomorphism/an anti-isomorphism from $\Gamma^a$ to $\Gamma^b$. For every $x\in T^a$ define a mapping $\pi^{(a, b)}:T^a\to T^b$  induced by $\psi$ as follows:

\textit{ Case 1:} $\Gamma^{a}\stackrel{\psi}{\sim}\Gamma^{b}$, $\psi$ is an isomorphism
  $$x\pi^{(a, b)}=\begin{cases}
 y \text{ such that } x'\psi=y' &\text{if } x'\in \Gamma^a, \\
z \text{ such that }  z\in T^b, z'\notin\Gamma^b, &\text{if } x'\notin \Gamma^a.
\end{cases}$$

\textit{Case 2:} $\Gamma^{a}\stackrel{\psi}{\sim}\Gamma^{b}$, $\psi$ is an anti-isomorphism
  $$x\pi^{(a, b)}=\begin{cases}
 y \text{ such that }   x'\psi=(y-1)' &\text{if } x'\in \Gamma^a,\\
z \text{ such that }  z\in T^b, (z-1)'\notin\Gamma^b, &\text{if } x'\notin \Gamma^a.
\end{cases}$$
\end{definition}
As the following example shows the mapping $\pi^{(a, b)}$ is not an isomorphism of $T^a$ to $T^b$: it does not  preserve the root in general.
 \begin{example}\label{ex_pi[a,b]} Consider order-preserving trees pictured in Fig.~\ref{fig:pi[a,b]}. Then $\pi^{(5, 1)}\colon T^5\to T^1$ and $\pi^{(5, 9)}\colon T^5\to T^9$ have the form
 $$\pi^{(5, 1)}=\left(
\begin{array}{ccccc}
1&2&3&4&5\\
1&2&3&4&5
\end{array}
\right),\ \ \ \pi^{(5, 9)}=\left(
\begin{array}{ccccc}
1&2&3&4&5\\
9&8&7&6&5
\end{array}
\right)$$

\begin{figure}[h]
\begin{center}
\begin{picture}(65,35)
\gasset{AHnb=0,linewidth=0.4}
  \drawline[linewidth=0.1](-5,10)(15,10)
  \drawline[linewidth=0.1](-5,15)(15,15)
  \drawline[linewidth=0.1](-5,20)(15,20)
 \drawline[linewidth=0.1](-5,25)(15,25)
  \drawline[linewidth=0.1](-5,30)(15,30)
  \drawline[linewidth=0.1](-5,35)(15,35)
  \put(-5,38){$1$}
  \put(0,38){$2$}
  \put(5,38){$3$}
  \put(10,38){$4$}
  \put(15,38){$5$}
  \drawline(-5,10)(-5,35)
  \drawline(0,10)(0,35)
  \drawline(5,10)(5,35)
  \drawline(10,10)(10,35)
  \drawline(15,10)(15,35)
  \drawline[linewidth=0.1](35,10)(75,10)
  \drawline[linewidth=0.1](35,15)(75,15)
  \drawline[linewidth=0.1](35,20)(75,20)
 \drawline[linewidth=0.1](35,25)(75,25)
  \drawline[linewidth=0.1](35,30)(75,30)
  \drawline[linewidth=0.1](35,35)(75,35)
  \put(34,38){$1$}
  \put(39,38){$2$}
  \put(45,38){$3$}
  \put(50,38){$4$}
  \put(55,38){$5$}
  \put(60,38){$6$}
  \put(65,38){$7$}
  \put(70,38){$8$}
  \put(75,38){$9$}
  \drawline(35,10)(35,35)
  \drawline(40,10)(40,35)
  \drawline(45,10)(45,35)
  \drawline(50,10)(50,35)
  \drawline(55,10)(55,35)
  \drawline(60,10)(60,35)
  \drawline(65,10)(65,35)
  \drawline(70,10)(70,35)
  \drawline(75,10)(75,35)
  \gasset{AHnb=0,linewidth=0.8, Nw=1.5,Nh=1.5,Nframe=y, Nfill=y, NLdist=3}
  \node(r1)(-5,35){}
  \node(r2)(0,25){}
   \node(r3)(5,20){}
  \node(r4)(10,15){}
  \node(r5)(15,30){}\put(17,31){$T^5$}
    \drawline(-5,10)(-5,35)
    \drawline(0,10)(0,25)
    \drawline(5,10)(5,20)
    \drawline(10,10)(10,15)
    \drawline(15,10)(15,30)
  \put(31,31){$T^1$}
  \node(r_21)(35,30){}
  \node(r_22)(55,35){}
  \node(r_23)(40,25){}
  \node(r_31)(45,20){}
  \node(r_32)(50,15){}
   \node(r_32)(60,15){}
  \node(r_33)(65,20){}
  \node(r_33)(70,25){}
  \node(r_33)(75,30){}\put(79,31){$T^9$}
    \drawline(35,10)(35,30)
    \drawline(40,10)(40,25)
    \drawline(45,10)(45,20)
    \drawline(50,10)(50,15)
    \drawline(55,10)(55,35)
    \drawline(60,10)(60,15)
    \drawline(65,10)(65,20)
    \drawline(70,10)(70,25)
    \drawline(75,10)(75,30)
  \end{picture}
\caption{Illustration of Example~\ref{ex_pi[a,b]}}\label{fig:pi[a,b]}
\end{center}
\end{figure}
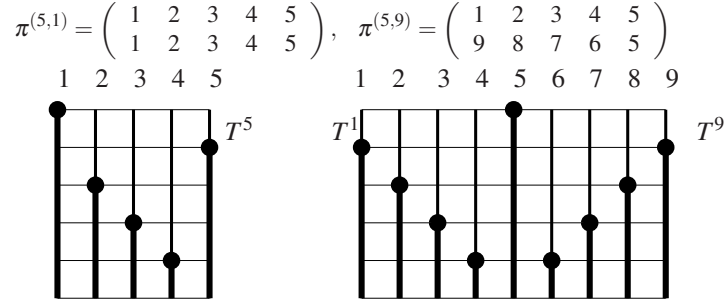
 \end{example}
Despite $\pi^{(n,n^\sharp)}$ is not an isomorphism, it has the following property:
 \begin{lem}\label{lem_property_pi} For all $x\in T^a, \p(a)\ne x\ne a$, it holds
 $\omega(x\pi^{(a,b)})=\omega(x)\pi^{(a,b)}.$
\end{lem}
\begin{proof} Let $T^a$ have $t$ levels. We induct on the level $i$ of $x\in T^a$, $2\leq i\leq t-1$. Clearly, $$\{a,\p(a)\}\pi^{(a,b)}=\{b,\p(b)\}.$$
 Therefore, if $x\in U_2$ then $$\omega(x)\pi^{(a,b)}=\{x,a,\p(a)\}\pi^{(a,b)}=\{x\pi^{(a,b)}, b,\p(b)\}=\omega(x\pi^{(a,b)}).$$
  Thus, the induction base holds. Suppose the statement holds for all $i<k$ for a natural $k$, $k<t$. Let $x\in U_k$ with the canonical bounds $y<x<z$. Without loss of generality assume $y$ is the parent of $x$. Let $[y',(z-1)']\psi=[p',(m-1)']$. Then $y\pi^{(a,b)}\in \{p,m\}$ and by the assumption $\omega(y)\pi^{(a,b)}=\omega(y\pi^{(a,b)})$. Suppose $[p',(l-1)']$, $[l',(m-1)']$ the children of $[p',(m-1)']$. Then on the one hand, by the construction of the inner tree $l$ is a child of $y\pi^{(a,b)}$. On the other hand, if $\psi$ is an isomorphism then $x'=l'$. If  $\psi$ is an anti-isomorphism then $x'=(l-1)'$. In both cases we get $x\pi^{(a,b)}=l$.  Thus,
  $$\omega(x)\pi^{(a,b)}=\{x\}\pi^{(a,b)}\cup\omega(y)\pi^{(a,b)}=\{l\}\cup \omega(y\pi^{(a,b)})=\omega(l)=\omega(x\pi^{(a,b)}).$$
\end{proof}

We also need to recall the following fact.
\begin{lem} [\cite{Bondar2016}, Lemma 5]\label{lem_image beta_tau} Let $L$, $\hat{L}$ be $\Lg$-cross-sections associated to similar respectful trees $\Gamma^a$, $\Gamma^b$, respectively . Let $\tau\colon L\to \hat{L}$ be an isomorphism. Then for all $A\in\Gamma^a$ and $\beta\in L$, it holds
$(A\beta)\psi=A\psi(\beta\tau)$.
\end{lem}
\begin{remark}\label{remark}It is convenient to interpret this result graphically, picturing the transformations like it was in Fig.~\ref{fig_dual_L}.  If $\psi$ is an isomorphism then we get the same depiction of transformation $\beta\tau\in \hat{L}$ as the $\beta$'s one (up to the labels of the vertices). If $\psi$ is an anti-isomorphism then $\Gamma^b$ is a mirror reflection of $\Gamma^a$. Therefore, to get the depiction of $\beta\tau$ it is enough to flip the $\beta$'s one upside-down.Consider an example.
\end{remark}
\begin{example}\label{ex_image beta_tau} Let $\Gamma_1$, $\Gamma_2$ be the trees from Example~\ref{ex_sim_trees}.  Apparently, $\psi$ is an anti-isomorphism, $\psi=\left(
\begin{array}{cccc}
1&2&3&4\\
4&3&2&1
\end{array}
\right)$. Without loss of generality consider $\beta \in L^{\Gamma_1}$ with $\im(\beta)=\{1,2,4\}$. by the definition of $\alpha_M$ we have $\beta=
\left(
\begin{array}{ccc}
\{12\}&\{3\}&\{4\}\\
1&2&4
\end{array}
\right).$ Then by Lemma~\ref{lem_image beta_tau}
$$\{12\}\psi(\beta\tau)=\{4,3\}\beta\tau,\ \{12\}\beta\psi=1\psi=4, \mbox{ whence } \{4,3\}\beta\tau=4.$$ Analogously
\begin{gather*}
\{3\}\psi(\beta\tau)=\{2\}\beta\tau,\ \{3\}\beta\psi=2\psi=3,\\
\{4\}\psi(\beta\tau)=\{1\}\beta\tau,\ \{4\}\beta\psi=4\psi=1.
\end{gather*}
Therefore, $\beta\tau=
\left(
\begin{array}{ccc}
\{1\}&\{2\}&\{34\}\\
1&3&4
\end{array}
\right)$ and its diagram is the  turned upside down $\beta$'s one (see Fig.~\ref{fig_image beta_tau}).

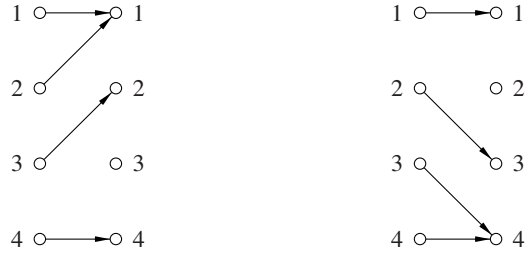
\begin{figure}[h]
\begin{picture}(60,45)
\gasset{AHLength=2.0,AHlength=2,AHangle=11}
\gasset{Nw=1.5,Nh=1.5,Nframe=y,Nfill=n,NLdist=3}
  \node[NLangle= 180](Na)(30,40){\small{$1$}}
  \node[NLangle= 180](Nb)(30,30){\small{$2$}}
  \node[NLangle= 180](Nc)(30,20){\small{$3$}}
  \node[NLangle= 180](Nd)(30,10){\small{$4$}}
  \node[NLangle= 0](N1)(40,40){\small{$1$}}
  \node[NLangle= 0](N2)(40,30){\small{$2$}}
  \node[NLangle= 0](N3)(40,20){\small{$3$}}
  \node[NLangle= 0](N4)(40,10){\small{$4$}}
\drawedge(Na,N1){}
\drawedge(Nb,N1){}
\drawedge(Nc,N2){}
\drawedge(Nd,N4){}
  \node[NLangle= 180](Na')(80,40){\small{$1$}}
  \node[NLangle= 180](Nb')(80,30){\small{$2$}}
  \node[NLangle= 180](Nc')(80,20){\small{$3$}}
  \node[NLangle= 180](Nd')(80,10){\small{$4$}}
  \node[NLangle= 0](N1')(90,40){\small{$1$}}
  \node[NLangle= 0](N2')(90,30){\small{$2$}}
  \node[NLangle= 0](N3')(90,20){\small{$3$}}
  \node[NLangle= 0](N4')(90,10){\small{$4$}}
\drawedge(Na',N1'){}
\drawedge(Nb',N3'){}
\drawedge(Nc',N4'){}
\drawedge(Nd',N4'){}
\end{picture}
\caption{ Illustration of Example~\ref{ex_image beta_tau}}\label{fig_image beta_tau}
\end{figure}
\end{example}
Now we are ready to prove the following lemma.

\begin{lem}\label{lem_sharp_psi} Let $(\overline{n},\prec_{1})=T^n$ and $(\overline{n},\prec_{2})$ be an elementary tree, $R\stackrel{\xi}{\cong}\hat{R}$. Then for all~$x$, $1<x<n$, it holds
 $$x^\sharp=x\pi^{(n, n^\sharp)}.$$
\end{lem}
\begin{proof} Let $\Gamma$, $\hat{\Gamma}$ be the  inner trees of $R$, $\hat{R}$ respectively; $L,\hat{L}$ be $\Lg$-cross-sections of $\On_{n-1}$ associated with $\Gamma$, $\hat{\Gamma}$. Let $\kappa\colon L\to \hat{L}:\alpha_M\mapsto \hat{\alpha}_M$ be the isomorphism defined in the proof of Proposition~\ref{prop_clasif_short R}. Since $(\overline{n},\prec_{1})=T^n$ it holds $\alpha_M^*\in R$ for all $\alpha_M\in L$. Then for $\hat\alpha_M\in \hat{L}$ we have $\hat\alpha_M^*=\alpha_M^*\xi$.

\par Let $x\in\overline{n}$, $1<x<n$. Let $\alpha_M\in L$ be such that $\alpha_M^*\in\Theta^x$. Then  $\alpha_M^*\xi\in\Theta^{x^\sharp}$.  Thus, if we  show that $\hat\alpha_M^*\in\Theta^{x\pi^{(n, n^\sharp)}}$  we get the required.

\par Since $\kappa$ is an isomorphism we can apply Lemma~\ref{lem_image beta_tau} to $\alpha_M, \hat\alpha_M$.
According to Proposition~\ref{prop_clasif_short R} we have to consider two possibilities:

\par \textit{Case 1.} $\Gamma\stackrel{\psi}{\sim}\hat{\Gamma}$, $\psi$ is an isomorphism. As we have observed  in Remark~\ref{remark} in this case $\alpha_M,$ $\hat\alpha_M$ has the same depiction. Since $x\pi^{(n, n^\sharp)}=x$ we get immediately $\omega(x\pi^{(n, n^\sharp)})=\omega(x)$ for all $x\in\overline{n}$, $1\ne x\ne n$. Thus, we have $\im(\hat\alpha_M^*)=\omega(x)=\omega(x\pi^{(n, n^\sharp)})$ as required.

\textit{Case 2.} $\Gamma\stackrel{\psi}{\sim}\hat{\Gamma}$, $\psi$ is an anti-isomorphism. Then the depiction of $\hat\alpha_M$ is obtained from the $\alpha_M$'s one if we consider the depiction upside down. Then clearly,  $\alpha_M^*$ and $(\hat\alpha_M)^*$ also have the symmetric diagrams with  $\ker{(\hat\alpha_M^*)}=(\ker{\alpha_M^*})\pi^{(n,n^\sharp)}$ and $\im{(\hat\alpha_M^*)}=(\im{\alpha_M^*})\pi^{(n,n^\sharp)}$. Thus, $\im{(\hat\alpha_M^*)}=\omega(x)\pi^{(n,n^\sharp)}$. By Lemma~\ref{lem_property_pi} it holds $\omega(x)\pi^{(n,n^\sharp)}=\omega(x\pi^{(n,n^\sharp)})$ for all $x\in\overline{n}, 1\ne x\ne n$. Therefore, $\hat\alpha_M^*\in\Theta^{x\pi^{(n,n^\sharp)}}$, whence $x^\sharp=x\pi^{(n, n^\sharp)}$ as required.
\end{proof}

At last, we can state the final form of the necessary condition of two $\R$-cross-sections being isomorphic.
\begin{lem}\label{lemma_class_neobhodimost'} Let $R\stackrel{\xi}{\cong}\hat{R}$; and the minimal fragmentations of $R,\hat{R}$ consist of $k$ elementary trees. Then for all $a_i\in \omega(R)$, $1\leq i\leq k$, and the inner trees $\Gamma^{a_i}$, $\Gamma^{a_i^\sharp}$ of $T^{a_i}$, $T^{a_i^\sharp}$ respectively, one of the following statements hold true:
\begin{itemize}
\item[(a)] $\Gamma^{a_i}\stackrel{\psi_i}{\sim}\Gamma^{a_i^\sharp}$, $\psi_i$ is an isomorphism.
\item[(b)] $\Gamma^{a_i}\stackrel{\psi_i}{\sim}\Gamma^{a_i^\sharp}$, $\psi_i$ is an anti-isomorphism.
\end{itemize}
\end{lem}
\begin{proof} We induct on $i$, $a_i\in \omega(R)$, $1\leq i\leq k$.

 Let $i=1$ and $a_1$ be the daughter of $\rt$. Then according to the proof of Proposition~\ref{prop_clasif_short R} we have $R(a_1)\cong R(a_1^\sharp)$ and $\Gamma^{a_1}\sim\Gamma^{a_1^\sharp}$ (see the proof of Proposition~\ref{prop_clasif_short R}).
 Assume $\psi_1$ is an isomorphism  for definiteness.

\par \textit{Step.} Let $a_i\in \omega(R,\rt_1)$, $1<i<k$.  By Proposition~\ref{prop_clasif_short R} it holds $\Gamma^{a_i}\stackrel{\psi_i}{\sim}\Gamma^{a_i^\sharp}$. If $|T^{a_i}|\leq3$ then clearly, the statement holds. Let $|T^{a_i}|>3$,  $u\in T^{a_i}$ with  $u\notin\{ \p(a_i), a_i, \dau(a_i)\}$.
\par Consider a transformation $\alpha\in R$ such that  $T^{a_i}$ is a transversal of $\overline{n}/\ker{(\alpha)}$.  Note that according to Definition~\ref{def_varphi}, $\alpha|_{T^{a_i}}$ defines the homomorphism from $T^{a_i}$ to $T^{a_1}$. Let $u\alpha=q$.

 We will show that there exists the homomorphism $T^{a_i^\sharp}\to T^{a_1}$ that maps $u\pi^{(a_i,a_i^\sharp)}$ to $q\pi^{(a_1,a_1^\sharp)}$. The condition clearly holds if and only if either both of $\psi_1$ and $\psi_i$  are the isomorphisms, or both of $\psi_1$ and $\psi_i$ are the anti- isomorphisms.

Note that $\im(\theta^u\alpha)=\omega(q)$. By Corollary~\ref{corol_tree_to_tree}, (3),  we have $\im(\theta^u\alpha)\xi=\omega(q^\sharp)$.  Therefore, $u^\sharp(\alpha\xi)=q^\sharp$. By the proof of Proposition~\ref{prop_clasif_short R} we have $R(a_i)\cong R(a_i^\sharp)$, thus by Lemma~\ref{lem_sharp_psi} we have $u^\sharp=u\pi^{(a_i,a_i^\sharp)}$, $q^\sharp=q\pi^{(a_1,a_1^\sharp)}$.

 Furthermore, it holds $\im(\theta^{x^\sharp}(\alpha\xi))=\im(\theta^x\alpha)^\sharp=\rt^\sharp$
  for all $x\in\omega(\p(a_i))$. Therefore by Definition~\ref{def_varphi} we have that the restriction $\alpha\xi|_{T^{a_i}}$ is the homomorphism from $T^{a_i^\sharp}$ to $T^{a_1^\sharp}$ and  $u\pi^{(a_i,a_i^\sharp)}(\alpha\xi)=q\pi^{(a_1,a_1^\sharp)}$ as required.
  \end{proof}

\subsection{The sufficient condition}

 In fact, the statement converse to Lemma~\ref{lemma_class_neobhodimost'} gives us the sufficient condition of isomorphic $\R$-cross-sections of $\On_n$:

\begin{lem}\label{lemma_suff_isom} Let $R$, $\hat{R}$ be the $\R$-cross-sections of $\overline{n}$ such that there exists an isomorphism (anti-isomorphism)
$\eta:\omega(R)\to \omega(\hat{R})$
and for all $a_i\in \omega(R)$ the inner trees $\Gamma^{a_i}, \Gamma^{a_i\eta}$ of elementary trees $T^{a_i}$, $T^{a_i\eta}$ are similar. Moreover, if $\Gamma^{a_i}\stackrel{\psi_i}{\sim}\Gamma^{a_i\eta}$ it holds either $\psi_i$ is isomorphism for all $a_i\in \omega(R)$, or $\psi_i$ is anti-isomorphism  for all $a_i\in \omega(R)$. Then $R\cong\hat{R}$.
\end{lem}
We divide the proof into few steps. First, we assign with each $\varphi\in R$ the transformation $\hat{\varphi}$. Then we investigate properties of $\hat{\varphi}$. After that we prove that $\xi:R\to\hat{R}:\varphi\to\hat{\varphi}$ is isomorphism.

Each transformation of an $\R$-section is determined by its partition. We construct the partition $\hat{\varphi} $ basing on the partition $\varphi $. In outlines, we will do the following: cut the partition of $\varphi$, change the resulting parts with $\psi_i$, and then glue the new ones back together.

Let $\widetilde{K}$ be the partition of $\varphi$. For a fixed $a_i\in\omega(R)$ by $\Lambda^{a_i}$ denote the kernel classes of the restriction $\varphi$  to the vertex set of the tree $T^{a_i}$:
$$\Lambda^{a_i}=\{X\cap T^{a_i}\mid X\in\widetilde{K}, X\cap T^{a_i}\ne\varnothing\}.$$

\par Now we define a partition $\Lambda^{a_i\eta}$ of a vertex set of $T^{a_i\eta}$ as follows:  $$\Lambda^{a_i\eta}=\{X\pi^{(a_i,a_i\eta)}\mid X\in\Lambda^{a_i}\},$$
where $X\pi^{(a_i,a_i\eta)}=\{x\pi^{(a_i,a_i\eta)}\mid x\in X\}$.

 Let $\Lambda\eta=\cup_{a_i\in\omega(R)}\Lambda^{a_i\eta}$. Note that for all $A,B\in \Lambda\eta$ it holds either $A\cap B=\emptyset$ or $A\cap B\in\omega(\hat{R})$.
 We define $\ker(\hat{\varphi})$ as follows: for all $x\in A, y\in B,$ let $$x\hat{\varphi}=y\hat{\varphi} \mbox{ if and only if either } A=B, \mbox{ or }  A\cap B\ne\emptyset.$$

 Further, let $x,y\in\omega(R)$ and $y$ have the child $z\in \omega(R)$. If $T^x\varphi=\{y\}$, $y\ne\rt$ then $T^x\varphi\subseteq T^{y}$ and $T^x\varphi\subseteq T^{z}$. To avoid the ambiguity we will write $T^x\varphi\subseteq T^{y}$ for the highest non-root vertex $y$ with this property. If $y=\rt$ we  will write formally $T^x\varphi\subseteq T^{\rt}$ to be consistent.
\begin{lem}\label{lemma_sv-vo_phihat} Let $R$, $\hat{R}$ satisfy the condition of Lemma~\ref{lemma_suff_isom}, $x\in\omega(R)$ with $x\ne \rt$. If $T^x\varphi\subseteq T^y$ for $y\in\omega(R)$ then $T^{x\eta}\hat{\varphi}\subseteq T^{y\eta}$, for all $x\in\omega(R)$.
 \end{lem}
 \begin{proof} Let $|\omega(R)|=t+1$. The proof is by induction on the level $i$ of $x\in\omega(R)$, $1\leq i\leq t$.

\par Let $x$ be the daughter of $\rt\in\omega(R)$. By the construction of the $\R$-cross-sections, there is only two possibilities for $T^x$: it holds either $T^x\varphi=\rt$, or $T^{x}\varphi\subseteq T^{x}$. If $T^x\varphi=\rt$ then by the construction of $\ker(\hat{\varphi})$, it holds also $|\Lambda^{x\eta}|=|\Lambda^{x}|=1$ therefore,  $(\rt\eta)\hat{\varphi}=(x\eta)\hat{\varphi}$, whence $T^{x\eta}\hat{\varphi}=\hat{\rt}$.
 If $T^{x}\varphi\subseteq T^x$ then $|\Lambda^{x\eta}|>1$. Consequently, $T^{x\eta}\varphi\subseteq T^{x\eta}.$

  \par Let  $x\in\omega(R)$ with level $k$, $0\leq k< t$. Suppose $T^{x}\varphi\subseteq T^{y}$ implies $T^{x\eta}\varphi\subseteq T^{y\eta}$. Let $d\in\omega(R)$ be the daughter of $x$. If $|T^{d}\varphi|=1$ then $x\varphi=d\varphi$, whence $T^{d}\varphi=\{y\}$. Moreover, since $|\Lambda^{d\eta}|=|\Lambda^{d}|=1$,  we have $(d\eta)\hat{\varphi}=(x\eta)\hat{\varphi}$. By the assumption $(x\eta)\hat{\varphi}=y\eta$. Therefore  $T^{d\eta}\hat{\varphi}\subseteq T^{y\eta}$ as required.

\par Let $|T^{d}\varphi|>1$. Then $T^{d}\varphi\subseteq T^{\dau(y)}$ and $|\Lambda^{d\eta}|=|\Lambda^{d}|>1$.
 Since $(x\eta)\hat{\varphi}=y\eta$ and $\eta:\omega(R)\to\omega(\hat{R})$ is isomorphism (anti-isomorphism), the definition of $\phi$ implies $(d\eta)\hat{\varphi}=(\dau(y))\eta$. Therefore, we get $T^{d\eta}\hat{\varphi}\subseteq T^{\dau(y)\eta}$  as required.
\end{proof}
\begin{lem}\label{lemma_hat_alpha}
For all $x\in\omega(R)$ the following statements holds:
\begin{itemize}
\item[(i)] If $T^x\varphi=\{x\varphi\}$  then $T^{x\eta}\hat{\varphi}=\{x\varphi\eta\}$.
\item[(ii)] If $|\Lambda^{x}|>1$ then for all $X\in \Lambda^{x}$ it holds
$X\pi^{(x,x\eta)}\hat{\varphi}=(X\varphi)\pi^{(x\varphi,x\varphi\eta)}$.
\end{itemize}
\end{lem}
 \begin{proof} Let $x\in\omega(R)$, $x\ne\rt$.
 Item (i) follows directly from the previous lemma.
\par Let $|T(\Lambda^x)|>1$ and $T^x\varphi\subseteq T^y$  with $x\varphi=y$, $y\in\omega(\hat{R})$.
By the description of $\R(\On_n)$ the mapping $T(\Lambda^{x})\to T^{y}: X\mapsto X\varphi$ defines a homomorphism of the trees.  By Definition~\ref{def_pi} and since $R$, $\hat{R}$ satisfy the condition of Lemma~\ref{lemma_suff_isom},  the mapping $T(\Lambda^{x\eta})\to T^{y\eta}: X\pi^{(x,x\eta)}\mapsto (X\varphi)\pi^{(y,y\eta)}$, $X\in\Lambda^x$, is also homomorphism. Therefore, by the description of $\R(\On_n)$ we get $X\pi^{(x,x\eta)}\hat\varphi=(X\varphi)\pi^{(y, y\eta)}$ as required.
 \end{proof}




Finally, we are ready to finish the proof of the sufficient condition.
\begin{lem} The mapping
$\xi:R\to \hat{R}:\varphi\mapsto \hat{\varphi}$ is an isomorphism.
\end{lem}
\begin{proof} It is clear that  $\xi$ is a bijective mapping. Let $\hat{\alpha}, \hat{\beta}\in \hat{R}$. Then for all $x\in\omega(R)$ and if $|\Lambda^x|=1$ according to Lemma~\ref{lemma_hat_alpha} we have
$$(x\eta)\widehat{\alpha\beta}=x(\alpha\beta)\eta=(x\alpha)\beta\eta=(x\alpha)\eta\hat{\beta}=x\hat{\alpha}\hat{\beta}.$$

Let $|\Lambda^x|>1$, $A\in \Lambda^x$. Denote by $a'$ an element of $A\pi^{(x,x\eta)}$. Then
\begin{gather*}a'\widehat{\alpha\beta}=A\pi^{(x,x\eta)}(\alpha\beta)\eta=
A(\alpha\beta)\pi^{[x(\alpha\beta),x(\alpha\beta)\eta]}=
(A\alpha)\beta\pi^{[(x\alpha)\beta,(x\alpha)\beta\eta]}=\\
=(A\alpha)\pi^{[x\alpha,x\alpha\eta]}\hat{\beta}=
(A\alpha\pi^{(x\alpha,x\alpha\eta)})\hat{\beta}=
A\pi^{(x,x\eta)}\hat{\alpha}\hat{\beta}=a'\hat{\alpha}\hat{\beta}.
\end{gather*}

 Therefore, $\xi$ is isomorphism as required.
\end{proof}

Lemmas~\ref{lemma_class_neobhodimost'} and~\ref{lemma_suff_isom} now give the classification of $\R$-cross-sections of $\On_n$ up to an isomorphism:
\begin{thm}\label{th_classification} The $\R$-cross-sections  $R, \hat{R}$ of $\On_n$ are isomorphic if and only if  there exists an isomorphism (anti-isomorphism) $\eta$ of the trees $\omega(R)$ and $\omega(\hat{R})$; and  for all $a_i\in \omega(R)$ with $a_i>\rt_1$ ($a_i<\rt_1$ respectively) and the inner trees $\Gamma^{a_i}$, $\Gamma^{a_i^\sharp}$ of $T^{a_i}$, $T^{a_i^\sharp}$ respectively, one of the following statements hold true:
\begin{itemize}
\item[(a)] $\Gamma^{a_i}\stackrel{\psi_i}{\sim}\Gamma^{a_i^\sharp}$, $\psi_i$ is an isomorphism.
\item[(b)] $\Gamma^{a_i}\stackrel{\psi_i}{\sim}\Gamma^{a_i^\sharp}$, $\psi_i$ is an anti-isomorphism.
\end{itemize}
\end{thm}

\end{document}